\documentclass[a4paper]{amsart}

\usepackage[usenames,dvipsnames]{pstricks}
\usepackage{epsfig}
\usepackage[space]{grffile} 
\usepackage{etoolbox} 
\makeatletter 
\patchcmd\Gread@eps{\@inputcheck#1 }{\@inputcheck"#1"\relax}{}{}
\makeatother

\usepackage[hypertex]{hyperref}
\usepackage{amsmath,amssymb}
\usepackage{bbold}
\usepackage{psfrag}
\usepackage{graphicx}
\usepackage{epstopdf}
\usepackage{subfigure}
\usepackage{leftidx}
\usepackage[all]{xy}
\usepackage[latin1]{inputenc}
\usepackage{enumerate}
\usepackage{caption}

\newtheorem{thm}{Theorem}[section]

\newtheorem{lem}[thm]{Lemma}
\newtheorem{prop}[thm]{Proposition}

\theoremstyle{definition}
\newtheorem{defi}[thm]{Definition}

\newtheorem{rem}[thm]{Remark}
\newtheorem{exa}[thm]{Example}

\newcommand{\Mod}{\mathrm{Mod}}

\newcommand{\Ar}{\mathrm{Ar}}

\newcommand{\ul}[1]{\underline{#1}}

\newcommand{\blabla}[1]{\quad\mbox{#1}\quad}
\newcommand{\co}{\colon}
\newcommand{\id}{\mathrm{id}}
\newcommand{\kk}{\Bbbk}

\newcommand{\Coalg}{\mathrm{Coalg}}

\newcommand{\C}{\mathcal{C}}
\newcommand{\B}{\mathcal{B}}
\newcommand{\D}{\mathcal{D}}
\newcommand{\A}{\mathcal{A}}

\newcommand{\K}{\mathcal{K}}

\newcommand{\U}{\mathcal{U}}
\newcommand{\V}{\mathcal{V}}
\newcommand{\Z}{\mathcal{Z}}

\newcommand{\M}{\mathcal{M}}

\newcommand{\re}{\textrm{r}}
\newcommand{\core}{\textrm{cor}}
\newcommand{\su}{\textrm{sum}}

\renewcommand{\lim}{\mathop\mathrm{lim}\limits}

\newcommand{\coeq}{\mathrm coeq}

\newcommand{\eq}{\mathrm eq}

\newcommand{\Rep}{\mathrm{Rep}}
\newcommand{\HRep}{\mathrm{HRep}}

\newcommand{\iso}{\stackrel{\sim}{\longrightarrow}}

\newcommand{\HM}{\mathcal{H}}

\newcommand{\Rt}{$\mathrm{R}$\nobreakdash-\hspace{0pt}}

\newcommand{\un}{\mathbb{1}}

\newcommand{\Ob}{\mathrm{Ob}}

\newcommand{\Aut}{\mathrm{Aut}}

\newcommand{\End}{\mathrm{End}}
\newcommand{\Hom}{\mathrm{Hom}}

\newcommand{\EndFun}{{\textsc{End}}}

\newcommand{\Fun}{{\textsc{Fun}}}
\newcommand{\BimodFun}{{\textsc{Bimod}}}
\newcommand{\coFun}{{\textsc{Fun}^{\textsc{co}}}}
\newcommand{\coBimodFun}{{\textsc{Bimod}^{\textsc{co}}}}
\newcommand{\coFunr}{{{\textsc{Fun}_r}^{\!\!\textsc{co}}}}
\newcommand{\coBimodFunr}{{{\textsc{Bimod}_r}^{\!\!\textsc{co}}}}

\newcommand{\vect}{\mathrm{vect}}

\newcommand{\Poly}{{\textsc{Poly}}}

\newcommand{\coPoly}{{\textsc{Poly}^{\textsc{co}}}}

\newcommand{\coinv}{\mathrm{co}}

\newcommand{\rdual}[1]{{#1}^\vee}

\newcommand{\adjunct}[2]{\!\!\raisebox{.6ex}{\xymatrix{\ar@/^.4pc/[r]^{#1} & \ar@/^.4pc/[l]^{#2}}}\!\!}

\providecommand{\bysame}{\leavevmode\hbox to3em{\hrulefill}\thinspace}

\newcommand{\eps}{\varepsilon}

\newcommand{\hT}{{\widehat{T}}}

\newcommand{\dar}[2]{\ar@<2pt>[r]^-{#1}\ar@<-2pt>[r]_-{#2}}
\newcommand{\labelto}[1]{\xrightarrow{#1}}

\usepackage[dvipsnames]{xcolor}

\newcommand{\tT}{\widetilde{T}}

\begin{document}
\title{Hopf polyads}
\author[A. Brugui\`eres]{Alan Brugui\`eres}
\email{alain.bruguieres@umontpellier.fr}

\date{\today}

\begin{abstract}
We introduce Hopf polyads in order to unify Hopf monads and group actions on monoidal categories. A polyad is a lax functor from a small category (its source) to the bicategory of categories, and a Hopf polyad is a comonoidal polyad whose fusion operators are invertible. The main result states that the lift of a Hopf polyad is a strong (co)monoidal action-type polyad (or strong monoidal pseudofunctor). The lift of a polyad is a new polyad having simpler structure but the same category of modules.
We show that, under certain assumptions, a Hopf polyad can be `wrapped up' into a Hopf monad. This generalizes the fact that finite group actions on tensor categories can be seen as Hopf monads. Hopf categories in the sense of Batista, Caenepeel and Vercruysse can be viewed as Hopf polyads in a braided setting via the notion of Hopf polyalgebras. As a special case of the main theorem, we generalize a description of the center of graded fusion category due to Turaev and Virelizier to tensor categories: if $\C$ is a $G$-graded (locally bounded) tensor category, then $G$ acts on the relative center of $\C$ with respect to the degree one part $\C_1$, and the equivariantization of this action is the center of $\C$.
\end{abstract}
\maketitle

\setcounter{tocdepth}{1} \tableofcontents

\newcommand{\RcoEq}{$\mathrm{E}$-}

\newcommand{\MOD}[2]{\mathrm{mod}_{#1}{#2}}

\section*{Introduction}

Hopf monads were introduced in \cite{BV2} as a generalization of Hopf algebras to the context of monoidal categories with duals (also called rigid, or autonomous categories), and subsequently extended in \cite{BLV} to arbitrary monoidal categories.
The initial motivation was to describe the center $\Z(\C)$ of a rigid category $\C$ in algebraic terms.
It turns out that, under reasonable assumptions,
$\Z(\C)$ is isomorphic as a braided category to the category of modules over a quasitriangular Hopf monad $Z_\C$ on $\C$. The monadicity part of this result was proved independently by Day and Street \cite{DayStreet}.
This description of the center, and the theory of the double of a Hopf monad developped in \cite{BV3}, were used by Turaev and Virelizier to compare the two main families of TQFTs for $3$-manifolds, namely the Turaev-Viro and the Reshetikin-Turaev TQFTs \cite{TV0}.

Hopf monads have also been applied in \cite{BN1} and \cite{BN2} to the study of tensor categories. In particular, the action of a finite group on a tensor category, and the corresponding equivariantization, can be interpreted in terms of Hopf monads. More precisely, let $G$ be a finite group acting on a tensor category $\C$, each $g \in G$ acting by a tensor autoequivalence $\rho_g$. Then $T = \bigoplus_{g \in G} \rho_g$ is a Hopf monad on $\C$, whose category of modules is the equivariantization of $G$ under the group action.
It should be noted, however, that group actions and equivariantizations cannot be interpreted in monadic terms in general. Indeed, monads arise in the presence of an adjunction, and in general, the forgetful functor of an equivariantization does not have a left adjoint, so that there is no associated monad.

In this article, we introduce the formalism of Hopf polyads in order to unify Hopf monads and group actions on monoidal categories. A polyad\footnote{After coming up with the name polyad, the author found out that a similar notion had been introduced by J. Benabou under the name \emph{polyades}.} $T$ is a lax functor from a small category $D$ (the \emph{source} of the polyad) to the bicategory of categories. It associates to every object $i$ of $D$ a category $\C_i$ and to every morphism $j \labelto{a} i$ of $D$, a functor $T_a :  \C_j \to \C_i$,
and is equipped with products $\mu_{a_b} : T_a T_b \to T_{ab}$ indexed by composable pairs $(a,b)$ of morphisms in $D$ and units $\eta_i : \id_{\C_i} \to T_{\id_i} = T_i$ indexed by objects $i$ of $D$, subject to associativity and unity axioms. In particular the endofunctor $T_i$ is a monad on $\C_i$ for each $i$. When $\mu$ and $\eta$ are isomorphisms, we say that $T$ is of action type (such a data is traditionally called a pseudofunctor). A monad is a special case of a polyad, obtained when $D$ is the final category $*$. On the other hand, if $G$ is a group, denoting by $\underline{G}$ the
category having one object $*$, with $\End_{\underline{G}}(*)=G$, then an action-type polyad with source $\underline{G}$ is nothing but an action of the group $G$ on a category $\C = \C_*$.

One defines modules over a polyad, which generalize modules over a monad. However the forgetful functor of the category of modules need not have a left adjoint - that is the main difference between polyads and monads.

The lift construction plays a special role. Given a polyad $T$, with source $D$, its lift $\tT$ is a new polyad with source $D$ canonically associated with $T$, which sends an object $i$
of $D$ to the category $\widetilde{\C}_i$ of $T_i$-modules. This construction can be carried out  when $T$ is \re-exact, that is, the categories $\C_i$ have reflexive coequalizers which are preserved by the functors $T_a$. The polyad $T$ and its lift $\tT$ have the same category of modules. The lift construction generalizes the classical algebraic fact that a bimodule can be lifted to a functor between categories of modules, which still holds in the monadic setting (Theorems~\ref{thm-eq-fun-bimod} and~\ref{thm-eq-comon-bimod}).

A comonoidal polyad is a polyad for which the categories $\C_i$ are monoidal, the functors $T_a$ are (lax) comonoidal, and the transformations $\mu_{a,b}$ and $\eta_i$ are comonoidal.
Given a comonoidal polyad, one may define a left and a right fusion operator $H^l$ and $H^r$, which generalize those of a comonoidal monad.
A Hopf polyad is a comonoidal polyad whose fusion operators are invertible. In particular, a Hopf monad is just a Hopf polyad with source $*$.

The central result of this article is the fundamental theorem of Hopf polyads (Theorem~\ref{thm-fond}). It asserts that the lift of a \re-exact transitive Hopf polyad is
a strong comonoidal polyad of action type. The transitivity condition is a technical way of saying that the functors $T_a$ are substantial enouggh (matter-of-factly : each endofunctor $T_a\un \otimes ?$ is conservative).
In a sense, this theorem means that the notion of a Hopf polyad is a conceptually economical way of unifying Hopf monads and strong (co)monoidal group actions and pseudofunctors.

We compare Hopf polyads and Hopf monads, and show that, under certain exactness conditions involving sums, a comonoidal polyad $T$ can be `wrapped up' into a comonoidal monad $\hT$ having same category of modules, and under stricter conditions (notably, that $D$ be a groupoid), $\hT$ is a Hopf monad if $T$ is a Hopf polyad. In the case
where $T$ is a a group action, $\hT$ is the associated Hopf monad constructed in \cite{BN1}.

We generalize Hopf modules to Hopf polyads in two ways, and give a decomposition theorem for both, derived from the analogue for Hopf monads given in \cite{BV2} and improved in \cite{BLV}.

As a source of examples, we show that Hopf categories, as defined by  Batista, Caenepeel and Vercruysse in \cite{BCV} can be viewed as special cases of Hopf polyads in a braided setting, via the notion of a
Hopf polyalgebra.

Lastly, returning to the initial motivation of understanding the center construction, we show that a result due to Turaev and Virelizier about the structure of the center of a graded fusion category can be understood as a special case of the fundamental theorem of Hopf polyads, and can be generalized to locally bounded graded tensor categories (Theorem~\ref{thm-center-graded}).
This result is based on the properties of a Hopf polyad $Z$ whose source is the category $\underline{G}$ of the group $G$, with, for $g \in G$:
$$
\left\{
\begin{array}{llll}
Z_g  : &\C &\to &\C,\\
&X &\mapsto &Z_g(X) = \int^{A \in \C_g} A \otimes X \otimes \rdual{A},\\
\end{array}
\right.$$
whose lift is an action of $G$ on the relative center $\Z_{\C_1}(\C)$. The center $\Z(\C)$ is the equivariantization of this action.

\subsection*{Organization of the text.}

In Section~\ref{sect-prelims}, we introduce notations and conventions, review standard facts about monoidal categories, monads, comonoidal functors, and outline the benefits of `\re-exactness conditions', that is, existence and preservation of reflexive equalizers. The most important contribution of this section is the notion of bimodule functors over monads, which are in a sense the main building block of the subsequent constructions. The key point is that, just as bimodules serve to encode right exact functors between categories of modules over rings, \re-exact bimodule functors serve to encode \re-exact functors between categories of modules over \re-exact monads (Theorem~\ref{thm-eq-fun-bimod}). The same holds for comonoidal \re-exact bimodule functors and comonoidal \re-exact functors between categories of modules over comonoidal \re-exact monads (Theorem~\ref{thm-eq-comon-bimod}).

In Section~\ref{sect-hopf-polyad}, we introduce the notions of a polyad, a comonoidal polyad and a Hopf polyad. We define the lift of a polyad, and state and prove the `fundamental theorem of Hopf polyads (Theorem~\ref{thm-fond}).

In Section~\ref{sect-fonct}, we define the $2$-categories of polyads and comonoidal polyads with source $D$, and the pull-back (source change). We study the special case of (transitive \re-exact) Hopf polyads whose source is a connected groupoid.

In Section~\ref{sect-quasitriangular}, we define \Rt matrices for comonoidal polyads and quasitriangular comonoidal polyads.

In Section~\ref{sect-polyads-monads}, we compare Hopf polyads and Hopf monads.

In Section~\ref{sect-hopfmodrep}, we generalize the notion of a Hopf module to Hopf polyads in two different ways, which leads to two decomposition theorems.

The subsequent sections are devoted to examples. Section~\ref{sect-hopf-cats} compares Hopf polyads and Hopf categories via the notion of Hopf polyalgebras. Hopf categories are special cases of Hopf polyalgebras, which in turn can be interpreted as an algebraic representation of so-called `representable Hopf polyads'.

Section~\ref{sect-center} introduces the Hopf polyad $Z$ associated with the center of a locally bounded $G$-graded tensor category $\C$, where $G$ is a group. The group $G$ operates on the relative center $\Z_{\C_1}(\C)$, the action being the lift of $Z$, and the center $\Z(\C)$ is monoidal equivalent to the equivariantization of this action.

\section{Preliminaries and notations}\label{sect-prelims}

\subsection{Categories, r-categories and r-exact functors}

Let $\C$ be a category. We denote by $\Ob(\C)$ the class of objects of $\C$ and by $\Hom_\C(X,Y)$ the set of morphisms in
$\C$ from an object $X$ to an object $Y$. The identity functor of~$\C$ is denoted by $1_\C$.

We will also use the following alternate notation: if $D$ is a small category, $D_0$ denotes the set of objects of $D$, $D_1$ the set of morphisms of $D$, and $D_2$ the set of composable pairs of $D$, that is,
pairs $(f,g)$ where $f, g \in D_1$ and the source of $f$ is the target of $g$.

We denote by $*$ the final category (with one object and one morphism). If $G$ is a group, we denote by $\ul{G}$ the category with one object $*$ whose endomorphisms are the elements of $G$. If $X$ is a set, we denote by $*_X$ the category whose set of objects is $X$, and having exactly one morphism in each Hom-set (which is equivalent to $*$ if $X \neq \emptyset$).
Lastly, we denote by $\Delta_1$ the category associated with the poset $\{0 \le 1\}$.


A pair of parallel morphisms $\xymatrix{X\dar{f}{g} &Y}$ in a category $\C$
is \emph{reflexive} 
 if $f$ and $g$ have a common section, 
 that is, if there exists a morphism $s \co Y \to X$ such that $fs=gs=\id_Y$.
A \emph{cofork} in $\C$ is a diagram
$$(\K) \quad\xymatrix{X \dar{f}{g} & Y \ar[r]^h &Z}$$
such that $hf = hg$. The cofork $(\K)$ is \emph{reflexive} if the pair $(f,g)$ is reflexive, and
it is \emph{exact} if $h$ is a coequalizer of the pair $(f,g)$.
A \emph{reflexive coequalizer} is a coequalizer of a reflexive pair.

On the other hand the cofork $(\K)$ is \emph{split} if there exist morphisms $\sigma : Z \to Y$, $\gamma : Y \to X$ such that $h\sigma = \id_Z$, $f\gamma = \sigma h$ and  $g\gamma = \id_Y$. Such a pair $(sigma,\gamma)$ is a \emph{splitting} of $(\K)$. A split cofork is exact, and indeed universally exact, as its image under any functor is again split, and therefore exact.

\begin{lem} A split cofork is exact.\end{lem}

\begin{proof} This is perfectly standard, but nonetheless instructive. Let $(\K)$ be a cofork as above, with splitting  $(\sigma,\gamma)$. Let $u : Z \to W$ be a morphism such that $uf = ug$.
If $u$ factorizes as  $vh$, then $v = vh\sigma = u\sigma$, so $v$ is unique. Now for $v = u\sigma$, $vh = u\sigma h = u f \gamma = u g \gamma = u$, which shows that $h$ is a coequalizer of $(f,g)$.
\end{proof}

Note that the image of a split cofork under any functor is again split, and therefore exact: split coforks are \emph{universally exact}.

\begin{defi}An \emph{\re-category} is a category admitting coequalizers of reflexive pairs.
A functor $F$ between \re-categories is \emph{\re-exact} if it preserves coequalizers of reflexive pairs.
\end{defi}

A functor $F: \C \to \D$ is \emph{conservative} if  any morphism $f$ of $\C$ such that $F(f)$ is an isomorphism of $\D$ is already an isomorphism of $\C$.

We will use the following well-known lemma.

\begin{lem} \label{lem-coforkex}
Let $F : \C \to \D$ be a conservative \re-exact functor between \re-categories. Let $(\K)$ be a reflexive cofork in $\C$ such that $F(\K)$ is exact. Then $(\K)$ is exact.
\end{lem}

\subsection{Monoidal categories}\label{sect-monocat}
Let $\C$ be a monoidal  category. Given an object $X$ of  $\C$, we denote by $X \otimes ?$ the endofunctor of $\C$ defined on objects by $Y \mapsto X \otimes Y$ and on morphisms by $f \mapsto X \otimes f=\id_X \otimes f$. The endofunctor $? \otimes X$ is defined similarly.


A monoidal \re-category is a monoidal category $\C$ admitting coequalizers of reflexive pairs, and whose tensor product $\otimes: \C \times \C \to \C$ is \re-exact. Note that this is equivalent to saying that for every object $X$ of $\C$, the endofunctors $X \otimes ?$ and $? \otimes X$ are \re-exact.

\begin{prop}\label{prop-faith}
Let $(C,\Delta,\eps)$ be a coalgebra in a monoidal \re-category. The following conditions are equivalent:

\begin{enumerate}[(i)]
\item the endofunctor $C \otimes ? : \C \to \C$ is conservative;
\item the endofunctor $? \otimes C : \C \to \C$ is conservative;
\item the cofork
$(\K) \quad \xymatrix{C \otimes C\dar{C \otimes \eps}{\eps \otimes C}&\C \ar[r]^{\eps}&\un}$
is exact.
\end{enumerate}

\end{prop}

\begin{defi}\label{def-trans} A coalgebra in a monoidal r-category is \emph{transitive} if it satisfies the equivalent conditions of Proposition~\ref{prop-faith}
\end{defi}

\begin{proof}
By reason of symmetry, it is enough to prove the equivalence of (i) and (iii).

Observe first that the cofork $C \otimes (\K)$ is split, with splitting $(\Delta, C \otimes \Delta)$: indeed, $(C \otimes \eps)\Delta = \id_C$,
$(C\otimes C \otimes \eps)(\Delta \otimes C) =\Delta (C \otimes \eps)$ and $(C \otimes \eps \otimes C)(\Delta \otimes C) = \id_{C \otimes C}$.
Therefore, $C \otimes (\K)$ is exact.

In particular, if $C \otimes ?$ is conservative, then by Lemma~\ref{lem-coforkex} $(\K)$ is exact, so (i) $\implies$  (iii). Conversely, assume $(\K)$ is exact and let $f : X \to Y$ be a morphism such that $C \otimes f$ is an isomorphism. We have a commutative diagram:
$$\xymatrix{
(\K) \otimes X & C \otimes C \otimes X \ar[d]_{C \otimes C \otimes f} \dar{}{} & C \otimes X \ar[d]_{C \otimes f}\ar[r] & X\ar[d]^f \\
(\K) \otimes Y & C \otimes C \otimes Y \dar{}{} & C \otimes Y \ar[r] & Y
}$$
whose lines are exact coforks, and since $C \otimes f$ and $C \otimes C \otimes f$ are isomorphisms, so is $f$. Hence $C \otimes ?$ is conservative. Thus (iii) $\implies$ (i).
\end{proof}

%
%

\subsection{Comonoidal functors}\label{sect-comonofunctor}
Let $(\C,\otimes,\un)$ and $(\D, \otimes, \un)$ be two monoidal categories.
A \emph{comonoidal functor} (also called \emph{opmonoidal functor}) from $\C$ to
$\D$ is a triple $(F,F^2,F^0)$, where $F\co \C \to \D$ is a functor, $F^2\co F \otimes \to F\otimes F$ is a natural
transformation, and $F^0\co F(\un) \to \un$ is a morphism in $\D$, such that:
\begin{align*}
&(F X \otimes F^2(Y,Z)) F^2(X,Y \otimes Z)= (F^2(X,Y) \otimes FZ) F^2(X \otimes Y, Z) \,\text{ for $X,Y,Z $ in $\C$};\\
&(F X \otimes F^0) F^2(X,\un)=\id_{F(X)}=(F^0 \otimes FX) F^2(\un,X) \,\text{ for $X$ in $\C$}.
\end{align*}
A comonoidal functor $(F,F^2,F^0)$ is \emph{strong} (resp.\@ \emph{strict}) \emph{comonoidal} if $F^2$ and $F^0$ are
isomorphisms (resp.\@ identities). In that case, $(F,{F^2}^{-1},{F^0}^{-1})$ is a strong (resp. strict) monoidal functor.

A natural transformation $\varphi\co F \to G$ between comonoidal functors is \emph{comonoidal} if it satisfies:
\begin{equation*}
G^2(X,Y) \varphi_{X \otimes Y}= (\varphi_X \otimes \varphi_Y) F^2(X,Y)\quad \text{and} \quad G^0 \varphi_\un= F^0.
\end{equation*}

Notions of monoidal functor and monoidal natural transformation are dual to the notions of comonoidal functor and comonoidal natural transformation.
As noted above, strong monoidal and strong comonoidal functors are essentially the same thing.

\subsection{Relative center}
Let $\C$, $\D$ be be monoidal categories and $F : \C \to \D$ be a comonoidal functor. A \emph{half-braiding relative to $F$} is a pair $(d,\sigma)$, where $d$ is an object of $\D$ and
$\sigma$ is a natural transformation $d \otimes F \to F \otimes d$ satisfying:
\begin{eqnarray*}
&(F^2(c,c') \otimes d)\sigma_{c \otimes c'}&= (F(c) \otimes  \sigma_{c'})(\sigma_c \otimes F(c'))(d \otimes F^2(c,c')),\\
&(F^0 \otimes d)\sigma_{\un}& = \sigma_\un (d \otimes
F^0).
 \end{eqnarray*}
Half-braidings relative to $F$ form a category called the
\emph{center of $\D$ relative to $F$}  and denoted by $\Z_F(\D)$,
or $\Z_\C(\D)$ if the functor $F$ is clear from the context. It is
monoidal, with the tensor product defined by $$(d,\sigma) \otimes
(d',\sigma') = (d \otimes d',  (\sigma \otimes d')(d \otimes
\sigma')),$$ and the forgetful functor $\U: \Z_{F}(\D) \to \D$
is monoidal strict. Note that for $F = \id_\D$, $\Z_F(\D)$ is the center $\Z(\D)$ of $\D$.

Now assume $F$ is strong monoidal (in particular, it can be viewed as a comonoidal functor).
Then we have a strong monoidal functor $\tilde{F}: \Z(\C) \to \Z_{F}(\D)$, defined by $\tilde{F}(c,\sigma) = (F(c),\tilde{\sigma})$, where $\tilde{\sigma} = F(\sigma)$ up to the structure isomorphisms of $F$.

\subsection{Monads}\label{sect-monads}

Let $\C$ be a category. The category $\EndFun(\C)$ of endofunctors of $\C$ is strict monoidal, the tensor product being the composition of endofunctors, and the unit, the identity functor.
A \emph{monad on $\C$} is an algebra in $\End(\C)$, that is, a triple $(T,\mu,\eta)$, where $T$ is an endofunctor of $\C$ and $\mu : T^2 \to T$, $\eta : \id_\C \to T$ are natural transformations satisfying:
$$\mu(T\mu) = \mu(\mu T) \quad\text{and}\quad \mu(T\eta) = \id_T = \mu(\eta T).$$

A \emph{$T$-module in $\C$} (often called a $T$-algebra) is a pair $(M,\rho)$, where $M$ is an object of $\C$ and $\rho$ is an action of $T$ on $M$, that is, a morphism $\rho : TM \to M$
satisfying:
$$\rho T(\rho)= \rho \mu_M \blabla{and} \rho T(\eta_M) = \id_M.$$

Morphisms between $T$-modules $(M,\rho)$ and $(N,\rho')$ are morphisms $f : M \to N$ in $\C$ satisfying $f \rho = \rho'T(f)$; with these morphisms, $T$-modules form a category denoted
by $\C_T$.

The \emph{forgetful functor} $U_T: \C_T \to \C$ is conservative. It has a left adjoint, namely the \emph{free module functor}
$$
\left\{
\begin{array}{llll}
F_T  : &\C &\to &\C_T,\\
&X &\mapsto &(TX,\mu_X),\\
&f &\mapsto &T(f).
\end{array}
\right.$$
The adjunction morphisms are $\eta: \id_\C \to U_T F_T = T$ and $\eps: F_T T \to \id_{\C_T}$e, defined by $\eps_{(M,\rho)} = \rho$.

Conversely, let $F \vdash U$ be an adjunction, that is, a pair of functors $U : \D \to \C$, $F : \C \to \D$, with adjunction morphisms
$\eta : \id_C \to UF$ and $\eps: FU \to \id_\D$, satisfying:
$$(U \eps)(\eta U) = \id_U \blabla{and} (\eps F)(F \eta) = \id_F.$$
Then $T=UF$ is a monad on $\C$, with product $\mu = U\eps F$ and unit $\eta$, so that we have a new adjunction $F_T \vdash U_T$. The two adjunctions are related by the \emph{comparison functor}
$$
\left\{
\begin{array}{llll}
K  : &\D &\to &\C_T,\\
&Y &\mapsto &(UY,U \eps_Y),\\
&f &\mapsto & Uf.
\end{array}
\right.$$
The adjunction $F \vdash U$ is \emph{monadic} if $K$ is an equivalence.

\subsection{Monadicity and \re-exactness}\label{sect-monadicity}

Let $\C$, $D$ be two categories.

A functor $U : \D \to \C$ is \emph{monadic} if it has a left adjoint $F$ and the comparison functor of the adjunction $(F \vdash U)$ is an equivalence.

Beck's theorem characterizes monadic functors in general. It can be restated in simple terms under \re-exactness conditions:

\begin{lem} (Weak monadicity)

Let $\C$, $\D$ be \re-categories and let $U : \D \to \C$ be an \re-exact functor. The following assertions are equivalent:
\begin{enumerate}[(i)]
\item $U$ is monadic;
\item $U$ is conservative and has a left adjoint $F$.
\end{enumerate}
If such is the case, the monad $T = UF$ is \re-exact.
\end{lem}

Monads which are \re-exact are well-behaved:

\begin{lem}
Let $\C$ be a \re-category, and let $T$ be a monad on $\C$. The following assertions are equivalent:
\begin{enumerate}[(i)]
\item $T$ is \re-exact;
\item $\C_T$ is a \re-category and $U_T$ is \re-exact.
\end{enumerate}
\end{lem}

\subsection{Module and bimodule functors}\label{sect-bimod-fun}

Let $\C$, $\D$ be two categories.

Given a monad $(T,\mu,\eta)$ on $\C$, a \emph{left action of $T$} on a functor $M: \D \to \C$ is a natural transformation $l : T M \to M$ satisfying
$$l(\mu \,M) = l(T \,l) \blabla{and} l \,(\eta \, M) = \id_M.$$
A \emph{left $T$-module functor with source $\D$} is a functor $M : \D \to \C$ endowed with a left action of $T$. Given two such left $T$-modules functors $(M,l)$ and $(M',l')$ with source $\D$, a natural transformation $f : M \to M'$ is \emph{left $T$-linear} if it satisfies: $f\,l = l' \, T f$.

Similarly, given a monad $(P,\mu',\eta')$ on $\D$, a \emph{right action of $T$} on a functor $M: \D \to \C$ is a natural transformation $r : M P \to M$ satisfying
$$r(M \mu') = r(r\, P) \blabla{and} r \,(M \,\eta') = \id_M.$$
A \emph{right $P$-module functor with target $\C$} is a functor $M : \D \to \C$ endowed with a right action of $P$.
Given two such right $P$-modules functors $(M,r)$ and $(M',r')$ with target $\C$, a natural transformation $f : M \to M'$ is \emph{right $T$-linear} if it satisfies: $f\,r = r' \, P f$.

Lastly, given both a monad $T$ on $\C$ and a monad $P$ on $\D$, a \emph{$T$-$P$-bimodule functor} is a functor $M : \D \to \C$ endowed with a
left action $r$ of $T$ and a right action $l$ of $P$ which commute, that is, which satisfy $l(T\, r) = r(l \, P)$.

We denote by $\BimodFun(T,P)$ the \emph{category of $T$-$P$-bimodule functors}, whose morphisms are natural transformations which are both $T$-linear on the left and $P$-linear on the right. When $P = \id_{\D}$ (resp. $T = \id_{\C})$ we obtain the \emph{category of left $T$-module functors with source $\D$} (resp. the \emph{category of right $P$-module functors with target $\C$}).

\begin{exa}
A $T$-module is nothing but a left $T$-module functor with source the initial category $*$, in other words:  $\C_T  = \BimodFun(T,\id_*)$.
\end{exa}

\begin{lem} Let $T$ be a monad on a category $\C$, let $P$ be a monad on a category $\D$, and let $F$ be a functor  $\D_P \to \C_T$. Then the functor $F^\flat = U_T F F_P$ has a canonical structure of $T$-$P$-bimodule functor.
\end{lem}

\begin{proof}
Define natural transformation $l=U_T \eps^T F F_P : U_T F_T U_T F F_P \to U_T F F_P$ and $r = U_T F \eps^P F_P:  U_T F F_P U_P F_P \to U_T F F_P$, where  $\eps^T : F_T U_T \to \id_{\C_T}$ and $\eps^P : F_P U_P  \to \id_{\D_P}$ are the adjunction coevaluations. One verifies that $l$ and $r$ are the left and right actions of a $T$-$P$-bimodule structure on the functor $F^\flat$.
\end{proof}

Consider a diagram of functors $\xymatrix{\D \ar[r]^N &\C \ar[r]^M& \B}$.
Let $(T,\mu,\eta)$ be a monad on $\C$, with a left action $l$ of $T$ on $N$
and a right action $r$ of $T$ on $M$.

We say that $M$ and $N$ are \emph{tensorable over $T$} if for every $X$ in  $\D$, the parallel pair $(r\,N, M\,l) : \xymatrix{MTN(X) \dar{}{} &MN(X)}$
admits a coequalizer. Note that this parallel pair is reflexive (having $M \eta_{NX}$ as common section), so that tensorability is ensured  whenever $\B$ is a \re-category.

When $M$ and $N$ are tensorable over $T$,  we define a functor
$$\left\{
\begin{array}{l}
   M \odot_T N : \D \to \B\\
  X \mapsto \coeq(\xymatrix{ MPN(X) \dar{r\,N}{M\,l} &MN(X)}),
\end{array}
\right.$$
called the \emph{tensor product of $M$ and $N$ over $T$}.

Note that if $\B$, $\C$, $\D$ are \re-categories and $T$, $M$ and $N$ are \re-exact, so is $M \odot_T N$.

The following lemma provides a key example of tensorable functors over a monad:

\begin{lem}\label{lem-exa-tensorable} Let $(T,\mu,\eta)$ be a monad on a category $\C$. Then $U_T : \C_T \to \C$ is a left $T$-module functor, and $F_T : \C \to \C_T$ is a right
$T$-module functor. Morever, $F_T$ and $U_T$ are $T$-tensorable, and we have a canonical isomorphism
$$F_T \odot_T U_T \iso \id_{\C_T}.$$
\end{lem}
\begin{proof}
The left action of $T$ on $U_T$ is $U_T \eps$, and its right action on $F_T$ is $\eps F_T$, where $\eps_{(M,\rho)} = \rho$.
Let $M$ be a $T$-module, with action $\rho : TM \to M$.  We have a reflexive cofork in $\C_T$:
$$\xymatrix{F_T T U_T M \dar{T \delta}{\mu_{M}} & F_T U_T M \ar[r]^\rho & M}.$$
A common section of the parallel pair is given by $F_T\eta_M$. Let us prove the exactness assertion. Let $N$ be a $T$-module, with action $\delta'$, and let
$f : F_T U_T M \to N$ be a $T$-linear morphism coequalizing the parallel pair. If $g : M \to N$ is such that $f= g\delta$, then $f \eta_M = g\delta \eta_M$ = $g$, hence the uniqueness of such a $g$.
Now setting $g = f \eta_M$, note first that $g \delta = f \eta_M \delta = f T(\delta) \eta_{TM} = f \mu_M \eta_{TM} = f$, as required, and moreover $g$ is $T$-linear:
$\delta' Tg = \delta' Tf T\eta_M = f \delta \eta_M = f = g \delta$. This shows the exactness of the above cofork. The coequalizer being unique up to unique isomorphism, we obtain a
canonical isomorphism: $F_T \odot_T U_T \iso \id_{\C_T}$.
\end{proof}

If in addition $P$ is a monad on $\D$, $S$ a \re-exact monad on $\B$, $M$ is a $S$-$T$-bimodule functor and $N$, a $T$-$P$-bimodule functor, then the left action of $M$ and the right action of $N$ define a structure of $S$-$P$-bimodule functor on $M \odot_T N$,
hence a functor
$$\left\{
\begin{array}{l}
 \BimodFun(S,T) \times \BimodFun(T,P) \to \BimodFun(S,P)  \\
  (M,N) \mapsto M \odot_T N
\end{array}
\right.$$

\begin{thm} \label{thm-eq-fun-bimod} Let $\C$, $\D$ be \re-categories and let $T$, $P$ be \re-exact monads on $\C$ and $\D$ respectively.
The assignment $F \mapsto F^\flat = U_T F F_P$ induces an equivalence of categories
$$?^\flat: \Fun_r(\D_P,\C_T) \to \BimodFun_r(T,P)$$
where $\Fun_r \subset \Fun$ denotes the full subcategory of $r$-exact functors, and $\BimodFun_r \subset \BimodFun$, the full subcategory of bimodules whose underlying functor is \re-exact.

A quasi-inverse of $?^\flat$ is given by $M \mapsto M^\sharp = M \odot_P ?$.
\end{thm}

\begin{proof}
In view of the above considerations, the functors $?^\flat$ and $?^\sharp$ are clearly well-defined, and all we have to do is check that they are quasi-inverses.
Denote by $\mu$ be the product of $P$ and by $\eta$ its unit.

Let $M$ be a \re-exact $T$-$P$-bimodule functor, and let $r : MP \to P$ be the right action of $P$ on $M$. We have $M^\sharp = M \odot_P ? : \D_P \to \C_T$, so ${M^\sharp}^\flat = U_T M\odot_P U_P$. Hence for $X \in \Ob(\D)$
an exact cofork:
$$\xymatrix{MP^2 X \dar{r P}{M \mu}& MP X\ar[r] & {M^\sharp}^\flat X}.$$
On the other hand, the cofork:
$$\xymatrix{MP^2 X\dar{r P}{M \mu}& MP X\ar[r]^r & MX}$$
is split, with splitting $(M\eta_X,MP \eta_X)$, and therefore exact.
By universality of the coequalizer, we have a canonical isomorphism $M X \iso {M^\sharp}^\flat X$.

Now let $F: \D_P \to \C_T$ be a \re-exact functor. For each $P$-module $M$, $${F^\flat}^\sharp M = F F_P\odot_P U_P M.$$
Now by Lemma~\ref{lem-exa-tensorable}, we have a canonical isomorphism: $F_P \odot_P U_P \iso \id_{\D_P}$,
and by composing this with $F$, which is \re-exact, we obtain a canonical isomorphism ${F^\flat}^\sharp M \iso F M$.
\end{proof}

The following lemma relates the composition of functors to the tensor product of bimodules.

\begin{lem}\label{lem-comp-tens} Let $\B$, $\C$, $\D$ be \re-categories and let $S$, $T$, $P$ be \re-exact monads on $\B$, $\C$, $\D$ respectively.
Let $M$ be a  \re-exact $S$-$T$-bimodule functor and let $N$ be a \re-exact $T$-$P$-bimodule functor.
We have a canonical isomorphism in $\Fun_\re(\D_P, \B_S)$:
$$(M \odot_{T} N)^\sharp \iso M^\sharp N^\sharp.$$
\end{lem}

\begin{proof}
Let $F = M^\sharp$ and $G= N^\sharp$.
According to Lemma~\ref{lem-exa-tensorable}, we have a canonical isomorphism: $$F_T \odot_T U_T \iso \id_{\D_T}.$$
Composing this isomorphism on the left with $U_S F$ (which is \re-exact) and on the right, with $G F_P$, we obtain an isomorphism $F^\flat \odot_{T} G^\flat \iso (FG)^\flat$,
and we conclude by observing that $M \simeq F^\flat$ and $N \simeq G^\flat$.
\end{proof}

\subsection{Comonoidal monads and Hopf monads}\label{sect-comohomo}

Let $\C$ be a monoidal category. A \emph{comonoidal monad} on $\C$ is a monad $(T,\mu,\eta)$ on $\C$ endowed with a comonoidal structure $(T,T^2,T^0)$ such that $\mu$ and $\eta$ are comonoidal
natural transformations.

If $T$ is a comonoidal monad on $^\C$, then $\C_T$ admits a unique monoidal structure such that $U_T: \C_T \to \C$ is comonoidal strict. The tensor product of $\C_T$ is defined by
$$(M,\delta) \otimes (M',\delta') = (M \otimes M', (\delta \otimes \delta')T^2(M,M')),$$
and its unit object is $(\un,T^0)$.

Comonoidal monads are related to comonoidal adjunctions. A \emph{comonoidal adjunction} is an adjunction $F \vdash U$, where $F : \D \to \C$, $U : \C \to \D$ are comonoidal functors between monoidal categories $\C$ and $\D$, and the adjunction morphisms $\eps : FU \to \id_\D$, $\eta : \id_\C \to UF$ are comonoidal. Note that these conditions imply that $U$ is actually strong comonoidal. Given such a comonoidal adjunction, $T= UF$ is a comonoidal monad on $\C$, and the comparison functor $K : \D \to \C_T$ is strong comonoidal. In particular, if the
adjunction is monadic, $K$ is a (strong) monoidal equivalence.

\begin{lem}
If $\C$ is a monoidal \re-category and $T$ a \re-exact comonoidal monad on $\C$, then $\C_T$ is a monoidal \re-category.
\end{lem}

The \emph{left and right fusion operators} $H^l$ and $H^r$ of a comonoidal monad $T$ on a monoidal category $\C$ are the natural transformations
$$H^l(X,Y) =  (T X \otimes \mu_Y){T^2}(X,TY) : T(X \otimes TY) \to T X \otimes T Y$$
and
$$H^r(X,Y) = (\mu_X \otimes T Y) {T^2}(TX,Y) : T(T X \otimes Y) \to T_X \otimes T_ Y$$
for $X$, $Y$ in $\C$.
A comonoidal monad is a \emph{left Hopf monad} (resp. \emph{right Hopf monad}) if its left (resp. right) fusion operator is an isomorphism. It is a \emph{Hopf monad} if both fusion operators are isomorphisms.

\subsection{Comonoidal bimodule functors}\label{sect-cobifun}

Let $T$ be a comonoidal monad on a monoidal category $\C$, and $P$ be a comonoidal monad on a monoidal category $\D$. A \emph{comonoidal $T$-$P$-bimodule functor} is a comonoidal functor
$M : \D \to \C$ endowed with a $T$-$P$-bimodule functor structure such that the left action $TM \to M$ and the right action $MP \to P$ are comonoidal natural transformations.

Comonoidal $T$-$P$-bimodule functors form a category $\coBimodFun(T,P)$, whose morphism are bimodule morphisms which are comonoidal natural transformations.

The results of Section~\ref{sect-bimod-fun} can be adapted in a straightforward way to comonoidal bimodule functors. Notably, if $F : \D_P \to \C_T$ is a comonoidal functor, then
$F^\flat = U_T F F_P :  _D \to \C$ is a $T$-$P$-comonoidal bimodule functor.

If $S$ is a \re-exact comonoidal monad on a monoidal category $\B$, $P$ a comonoidal  monad on $\D$, $M$ is a comonoidal $S$-$T$-bimodule functor and $N$, a comonoidal $T$-$P$-bimodule functor, then $M \odot_T N$ is a comonoidal $S$-$P$-bimodule functor,
hence a functor
$$\left\{
\begin{array}{l}
 \coBimodFun(S,T) \times \coBimodFun(T,P) \to \coBimodFun(S,P)  \\
  (M,N) \mapsto M \odot_T N
\end{array}
\right.$$

\begin{thm} \label{thm-eq-comon-bimod} Let $\C$, $\D$ be monoidal \re-categories and let $T$, $P$ be \re-exact comonoidal monads on $\C$ and $\D$ respectively.
The assignment $F \mapsto F^\flat = U_T F F_P$ induces an equivalence of categories
$$?^\flat: \coFunr(\D_P,\C_T) \to \coBimodFunr(T,P)$$
where $\coFunr \subset \coFun$ denotes the full subcategory of $r$-exact comonoidal functors, and $\coBimodFunr \subset \coBimodFun$, the full subcategory of comonoidal bimodules whose underlying functor is \re-exact.

A quasi-inverse of $?^\flat$ is given by $M \mapsto M^\sharp = M \odot_P ?$.
\end{thm}

\begin{lem}\label{lem-comon-tens} Let $\B$, $\C$, $\D$ be monoidal \re-categories and let $S$, $T$, $P$ be comonoidal \re-exact monads on $\B$, $\C$, $\D$ respectively.
Let $M$ be a \re-exact comonoidal $S$-$T$-bimodule functor, and $N$ be a \re-exact comonoidal $T$-$P$-bimodule functor.
We have a canonical comonoidal isomorphism in $\coFun(\D_P,\B_S)$:
$$(M \odot_{T} N)^\sharp \iso M^\sharp N^\sharp.$$
\end{lem}

\section{Hopf Polyads}\label{sect-hopf-polyad}

\subsection{Polyads}

A \emph{polyad} consists in a set of data $(D,\C, T, \mu,\eta)$, where

\begin{enumerate}[1)]

\item $D$ is a category,

\item $\C = (\C_i)_{i \in D_0}$ is a family of categories indexed by the objects of $D$,

\item $T = (T_a)_{a \in D_1}$ is a family of functors indexed by the morphisms of $D$, whith $T_a : \C_j \to \C_i$ for $a \in \Hom_D(j,i)$,

\item $\mu = (\mu_{a,b})_{(a,b) \in D_2}$ is a collection of natural transformations indexed by composable pairs of morphisms of $D$, with $\mu_{a,b} : T_aT_b  \to T_ab$,

\item $\eta = (\eta_i)_{i \in D_0}$ is a collection of natural transformations indexed by objects of $D$, with $\eta_i : \id_{\C_i} \to T_i = T_{\id_i}$ for $i \in D_0$,

\end{enumerate}
subject to the following  axioms :
\begin{enumerate}[(i)]
\item Given three composable morphisms $a, b, c$ in $D$,
$$\mu_{ab,c} (\mu_{a,b} \,\id_{T_c}) = \mu_{a,bc}(\id_{T_a} \,\mu_{b,c}),$$
\item Given a morphism $a : j \to i$ in $D$ :
$$\mu_{a,\id_j}(\id_{T_a} \,\eta_j) = \id_{T_a} = \mu_{id_i,a}(\eta_i \,\id_{T_a}).$$
\end{enumerate}

If $(D,\C, T, \mu,\eta)$ is a polyad, the category $D$ is called the \emph{source} of $T$.
A polyad with source $D$ is also called a $D$-polyad.

We denote by $|T|$ the category $\prod_{i \in D_0} \C_i$.

The polyad $T$ is said to be \emph{of action type} if the natural transformations $\mu$ and $\eta$ are isomorphisms.

\begin{rem}
A $D$-polyad is nothing but a lax functor from $D$ to the bicategory of categories. It is of action type if it is a pseudofunctor.
\end{rem}

\begin{exa}
A $*$-polyad $T$, where $*$ denotes the initial category, is nothing but a monad $T$ on a category $\C$.
\end{exa}

\begin{exa}
A $\Delta_1$-polyad, where $\Delta_1$ denotes the category associated with the poset $\{0,1\}$, is a bimodule functor, that is, it consists in the following data:
a monad $T_0$ on a category $\C_0$, a monad $T_1$ on a category $\C_1$, and a $T_1$-$T_0$-bimodule functor $M$.
\end{exa}

\begin{exa}\label{exa-act}
If $G$ is a group, a action-type $\ul{G}$-polyad $T$ is nothing but an action of the group $G$ on a category $\C$.
\end{exa}

\subsection{Modules over a polyad} Let $(D,\C, T, \mu,\eta)$ be a polyad. A \emph{$T$-module} consists in a pair $(X,\rho)$ as follows:
\begin{enumerate}[1)]
\item  $X = (X_i)_{i \in D_0}$ is an object of $|T|$,
\item  $\rho= (\rho_a)_{a \in D_1}$, where for $j \labelto{a} i$ in $D_1$, $\rho_a$ is a morphism $T_a(X_j) \to X_i$ in $\C_i$,
\end{enumerate}
subject to the following axioms:
\begin{enumerate}[(i)]
\item given $(a,b) \in D_2$, $b: k \to j$,  $a : j \to i$, the following diagram commutes:
$$\xymatrix{
T_aT_b(X_k) \ar[r]^{T_a(\rho_b)} \ar[d]_{{\mu_{a,b}}_{X_k}}& T_a(X_j) \ar[d]^{\rho_a}\\
T_{ab}(X_k) \ar[r]_{\rho_{ab}}& X_i
}$$
\item given $i \in D_0$, $\rho_i \eta_i = \id_{X_i}$.
\end{enumerate}

Given a polyad $T$, $T$-modules form a category which we denote by $\Mod(T)$. A morphism between $T$-modules $(X,\rho)$ and $(X',\rho')$ is a morphism $f : X \to X'$ in $|T|$ such that for any $a : j \to i$ in $D_1$,
the following square commutes:
$$\xymatrix{
T_a(X_j) \ar[r]^{T_a(f_j)} \ar[d]_{\rho_a}& T_a(X'_j)\ar[d]^{\rho'_a}\\
X_i \ar[r]_{f_i} & X'_i\,.
}$$

Note that the forgetful functor $U_T : \Mod(T) \to |T|$ is conservative, but in general it does not have a left adjoint. See Section~\ref{sect-polyads-monads} for a discussion of this point.

\begin{exa}\label{exa-constant}
Let $D$ be a small category and $\V$ be an arbitrary category. Define a $D$-polyad $\V_D = (D,\V^{D_0},T=(\id_\V)_{a \in D_1},\mu =\id,\eta = \id)$.
Then $\Mod(D) = \Fun(D,\V)$. A polyad of the form $\V_D$ is a \emph{constant polyad}.
\end{exa}

\begin{lem}\label{lem-act-inv}
Let $T$ be an action-type polyad, and $(X, \rho)$ be a $T$-module. If $a \in D_1$ is invertible, then $\rho_a$ is an isomorphism.
\end{lem}

\begin{proof} For $i \in D_0$, $\rho_i \eta_i = \id_{X_i}$ and since $\eta_i$ is an isomorphism, so is $\rho_i$.
For $j \labelto{a} i$, $\rho_a (T_a \rho_{a^{-1}}) =  \rho_{i} {\mu_{a,a^{-1}}}_X$ is an isomorphism, so $\rho_a$ has a right inverse.
Now by symmetry, $\rho_{a^{-1}}$ has a right inverse too, as well as $T_a \rho_{a^{-1}}$. Since $\rho_a$'s right inverse is invertible, $\rho_a$ itself is invertible.
\end{proof}

\begin{exa}
\label{exa-equiv}
An action of a group $G$ on a category $\V$  can be viewed as an action-type $\ul{G}$-polyad $T$, and $\Mod(T)$ is the equivariantization of
the category $\C$ under the action of $G$. If $(X,\rho)$ is a $T$-module, $\rho$ is an isomorphism.
\end{exa}

\subsection{Representations of a polyad}\label{sect-rep-poly}

There is an alternate, `bigger' notion of modules over a polyad, which we call representations. Let $T = (D,\C,T,\mu,\eta)$ be a polyad. A \emph{representation of $T$}
is a set of data $(W,\rho)$, where :
\begin{enumerate}[(1)]
\item $A = (W_a)_{a \in D_1}$, with $A_a \in \Ob(\C_j)$ for $j \labelto{a} i$ in $\D_1$,
\item $\rho= (\rho_{a,b})_{(a,b) \in D_2}$, where $\rho_{a,b}$ is a morphism $T_a W_b \to W_{ab}$ for $k \labelto{b} j \labelto{a} i$ in $D_2$,
\end{enumerate}
with the following axioms:
\begin{enumerate}[(1)]
\item for $l \labelto{c} k \labelto{b} j \labelto{a} i$ composable morphisms in $D$,
$$\rho_{ab,c}{\mu_{a,b}}_{W_c} = \rho_{a,bc} (T_a\,\rho_{b,c}),$$
\item for  $j \labelto{a} i$ in $\D_1$, $$\rho_a {\eta_i}_{W_a} = \id_{W_a}.$$
\end{enumerate}

Denote by $\Rep(T)$ the category of representations of $T$, a morphism between two representations $(W,\rho)$ and $(W',\rho')$ being a family of morphisms $(f_a : W_a \to W'_a)_{a \in D_1}$
such that $\rho'_{a,b}(T_a f_b) = f_{ab} \rho_{a,b}$ for $(a,b) \in D_2$.


For $(W,\rho)$ representation of $T$, define an object $V_T(W,\rho)$ of $|T|$ by  $$V_T(W,\rho) = (W_i)_{i \in D_0} \in \Ob (|T|).$$

Convesely, for $X \in |T|$, define a representation $L_T X$ of $T$ by
$$L_T X = ((T_a X_{s(a)})_{a \in D_1}, ({\mu_{a,b}}_{X_{s(b)}})_{(a,b) \in D_2}).$$

\begin{thm}\label{thm-rep}
The assignments $(W,\rho) \mapsto W_T(W,\rho)$ and $X \mapsto L_T X$ define functors:
$$L_T : |T| \to \Rep(T) \blabla{and} V_T : \Rep(T) \to |T|.$$
Moreover, $L_T$ a left adjoint of $V_T$, with adjunction morphisms
\begin{align*}&\eta : \id_{|T|} \to V_T L_T, &\eps : L_T V_T \to \id_{\Rep(T)},\\
&\eta_X  = ({\eta_i}_{X_i})_{i \in D_0}, &\eps_{(W,\rho)} = (\rho_{a,s(a)})_{a \in D_1}.
\end{align*}
In particular, if $T$ is of action type and $D$ is a groupoid, then $U_T : \Rep(T) \to |T|$ is an equivalence.
\end{thm}

\begin{proof} By straightforward computation. The last observation deserves an explanation. If $T$ is of action type, then $\eta$ is an isomorphism, hence $L_T$ quasi-inverse to $V_T$ on the right.
If in addition $D$ is a groupoid, $\eps$ is invertible because the morphisms $\rho_a$ are invertible, hence $L_T$ is also quasi-inverse to $V_T$ on the left. The inverse of $\rho_{a,b} : T_a W_b \to W_{ab}$ is the compositum of the morphisms:
$$W_{ab} \iso T_a T_{a^{-1}} W_{ab} \labelto{T_a \rho_{a^{-1},ab}} T_a W_b.$$

\end{proof}

\begin{rem} Representations of polyads correspond with the representations of a Hopf category
considered in \cite{BCV}. For a monad,that is, a $*$-polyad, representations and modules coincide. They differ in general, for instance if $T$ is the polyad of the action of a group $G$ on a category $D$, $\Mod(T)$ is the equivariantization, whereas $\Rep(T)$ is equivalent to $\C$ by Theorem~\ref{thm-rep}. However representations can be seen as a modules via the notion of pull-back, see
Lemma~\ref{lem-mod-rep}.
Since the focus of the present work is on group actions, equivariantizations and generalizations thereof,
we will prefer modules to representations. In Section~\ref{sect-hopfmodrep}, we will see that Hopf modules can be generalized to the framework of polyads in two ways, either as modules endowed with a comodule structure (Hopf modules), or as representations endowed with a comodule structure (Hopf representations). The decomposition theorem for Hopf modules extends to both of these notions, but the Hopf representation version is more natural.
\end{rem}

\subsection{The lift of a polyad}

A polyad $T = (D,\C, T, \mu,\eta)$ is said to be $r$-exact if for each $i \in D_0$, $\C_i$ is a \re-category (\emph{i.e.} admits reflexive coequalizers) and for each $a : j \to i$ in  $D_1$, the functor $T_a : \C_j \to \C_i$ is \re-exact (\emph{i.e.} preserves reflexive coequalizers).

Let $T=(D,\C,T,\mu,\eta)$ be a \re-exact polyad. For $i \in D_0$, $(T_i,\mu_{i,i},\eta_i)$ is a r-exact monad on $\C_i$. Set $\tilde{\C_i} = {\C_i}_{T_i}$.
For $a : j \to i \in D_1$, $T_a : \C_j \to \C_i$ is a r-exact $T_i$-$T_j$-bimodule functor, and in view of Theorem~\ref{thm-eq-fun-bimod}, it lifts to a
r-exact functor $(T_a)^\sharp = T_a \odot_{T_j} ? \,: \tilde{\C}_j \to \tilde{\C}_i$.
Set $\tT_a = (T_a)^\sharp$.

By Lemma~\ref{lem-comp-tens}, for $(a,b) \in D_2$ we have $\tT_a \tT_b = (T_a \odot_{T_j} T_b)^\sharp$, so that $\mu_{a,b} : T_a T_b \to T_{ab}$ induces a natural transformation
$\tilde{\mu}_{a,b} : \tT_a \tT_b \to \tT_{ab}$.
Similarly, for $i \in D_0$, $\eta_i: \id_{\C_i} \to T_i$ induces a natural transformation $\tilde{\eta}_i : \id_{{\C_i}_{T_i}} \to \tT_i$.

\begin{prop} Let $T = (D,\C,T,\mu,\eta)$ be a \re-exact polyad. Then with the above notations $\tT = (D,\tilde{\C},\tT,\tilde{\mu},\tilde{\eta})$
is an \re-exact polyad. Moreover, $\tilde{\eta}_i$ is the identity for $i \in D_0$.
\end{prop}

\begin{defi}The polyad $\tT$ is called the \emph{lift of the polyad $T$}. \end{defi}

\begin{proof}
The axioms of a polyad for $\tT$ result from the axioms for $T$, since the structure of $\tT$ is induced by that of $T$ by quotienting.
Moreover $\tT_i$ is the identity monad of ${\C_i}_{T_i}$, and in particular
$\tilde{\eta}_i$ is the identity.
\end{proof}


Given a $\tT$-module $(X,\rho)$, we define a $T$-module $(X^\flat,\rho^\flat)$ as follows:
for $i \in D_0$, $X^\flat_i = U_{T_i}(X_i)$, and for $j \labelto{a} i$ in $D_1$,
$\rho^\flat_a$ is the compositum of
$$T_a U_{T_j} X_j \iso U_{T_i} \tT_a F_{T_j} U_ {T_j} X_j \labelto{\tT_a(\rho_j)} U_{T_i} \tT_a X_j \labelto{\rho_a} U_{T_i} X_i.$$

\begin{prop} If $T$ is an \re-exact polyad, the assignment $X \mapsto X^\flat$ defines an isomorphism of categories $$\Mod(\tT) \iso \Mod(T).$$
\end{prop}

\begin{proof}
If $(Y,r)$ is a $T$-module, define a $\tT$-module $(Y^\sharp,r^\sharp)$ as follows. For $i \in D_0)$,  $Y^\sharp_i$ is the $T_i$-module $(Y_i,r_i)$.
For $j \labelto{a} i$ in $D_1$, the morphism $r_a: T_a Y_j \to Y_i$ induces a morphism $r^\sharp_a$ of $T_i$-modules $\tT_a Y^\sharp_j = T_a \odot_{T_j} Y^\sharp_j \to Y^\sharp_i$.
One verifies that the assignment $Y \to Y^\sharp$ so defined is inverse to the assignment $X \mapsto X^\flat$.
\end{proof}

\subsection{Comonoidal polyads}

A \emph{comonoidal polyad} is a polyad $T=(D,\C,T,\mu,\eta)$ endowed with a monoidal structure on the category $\C_i$ for each $i \in D_0$, and of a comonoidal structure on the functor $T_a$ for each morphism $a$ in $D$, in such a way that $\mu$ and $\eta$ are comonoidal, that is, the natural transformations $\mu_{a,b}$ for $(a,b) \in D_2$ and $\eta_i$ for $i \in D_0$ are comonoidal. For instance, the compatibility between the product and the comonoidal structure is expressed by the commutativity of the diagram:
$$
\xymatrix @C=5pc{
T_a T_b(X\otimes Y)\ar[d]_{{\mu_{a,b}}_{X \otimes Y}} \ar[r]^{T_aT^2_b(X,Y)}& T_a(TbX \otimes T_bY) \ar[r]^{T^2_a(T_bX,T_bY)} & T_a T_b X \otimes T_a T_b Y \ar[d]^{{\mu_{a,b}}_X \otimes {\mu_{a,b}}_Y}\\
T_{ab}(X \otimes Y) \ar[rr]_{T^2_{ab}(X,Y)} && T_{ab}X \otimes T_{ab} Y.
}
$$
for $k \labelto{b} j \labelto{a} i$ in $D_2$ and $X,Y$ in $\C_k$.

\begin{defi} A comonoidal polyad $T$ is said to \emph{preserve tensor products} (resp. \emph{tensor units}) if for each $a \in D_1$, the comonoidal functor $T_a$ preserves tensor products (resp. tensor units). A \emph{strong comonoidal polyad} is a comonoidal polyad which preserves tensor products and tensor units.
\end{defi}

 A comonoidal polyad $T = (D,\C,T,\mu,\eta)$ is \emph{\re-exact} if it is \re-exact as a polyad and the categories $\C_i$ are monoidal \re-categories (that is, their tensor products are \re-exact in each variable). If such is the case, $T$ is \emph{transitive} if for each $j \labelto{a} i$ in $D_1$, $T_a \un$ is a transitive coalgebra (see Definition~\ref{def-trans}) in the monoidal r-category $\C_i$.

\begin{prop}\label{prop-comopo} Let $T$ be an \re-exact comonoidal polyad. Then
\begin{enumerate}[(1)]
\item $\tT$ is an \re-exact comonoidal  polyad;
\item the canonical equivalence of categories $?^\flat : \Mod(\tT) \to \Mod(T)$ is a srtict monoidal equivalence.
\end{enumerate}
\end{prop}

\begin{proof}
Results from the basic properties of comonoidal monads and comonoidal bimodule functors, see Sections~\ref{sect-comohomo} and~\ref{sect-cobifun}.
\end{proof}

\begin{defi}
Given a comonoidal polyad $T =(D,\C,T,\mu,\eta)$, the \emph{left (resp. right) fusion operator of $T$}, denoted $H^l$ (resp $H^r$) is a collection of natural tranformations indexed by  composable pairs $(a,b) \in D_2$,
 $k \labelto{b} j \labelto{a} i$, defined as follows:
$$H^l_{a,b}(X,Y) =  (T_a X \otimes {\mu_{a,b}}_Y){T_a^2}(X,T_b Y) : T_a(X \otimes T_b Y) \to T_a X \otimes T_{ab} Y$$
and
$$H^r_{a,b}(Y,X) = ({\mu_{a,b}}_Y \otimes T_a X) {T_a^2}(T_b Y,X) : T_a(T_b Y \otimes X) \to T_{ab} Y \otimes T_a X$$
for $X$, $Y$ objects of $\C_k$ and $\C_j$ respectively.
\end{defi}

\begin{prop}\label{prop-fusion} The fusion operators of a comonoidal polyad $T$ satisfy:
\begin{enumerate}[(1)]
\item for $l \labelto{c} k \labelto{b} j \labelto{a} i$ in $D$ and $X \in \Ob(\C_j), Y \in \Ob(\C_l)$:
$$H^l_{a,bc}(X,Y)T_a(X \otimes {\mu_{b,c}}_Y) = (T_a X \otimes {\mu_{ab,c}}_Y)H^l_{a,b}(X,T_c Y),$$
$$H^r_{a,bc}(Y,X)T_a({\mu_{b,c}}_Y \otimes X) = ({\mu_{ab,c}}_Y \otimes T_a X)H^r_{a,b}(T_c Y,X),$$
\item for $j \labelto{a} i$ in $D_1$, $X, Y \in \Ob(\C_i)$:
$$H^l_{a,j}(X,Y)T_a(X \otimes {\eta_j}_{Y})=T^2_a(X,Y);$$
$$H^r_{a,j}(X,Y)T_a({\eta_j}_{X} \otimes Y)=T^2_a(X,Y);$$
\item for $j \labelto{a} i$ in $D_1$, $X\in \Ob(\C_i), Y \in \Ob(\C_j)$:
$$H^l_{i,a} {\eta_i}_{X \otimes T_aY} = {\eta_i}_X \otimes T_a Y;$$
$$H^r_{i,a} {\eta_i}_{T_aY \otimes X} = T_a Y \otimes {\eta_i}_X;$$
\item for  $k \labelto{b} j \labelto{a} i$ in $D$ and $X; X' \in \Ob(\C_j), Y \in \Ob(\C_k)$:
$$(T^2_a(X,X') \otimes T_{ab}Y)H^l_{a,b}(X \otimes X',Y) = (T_a X \otimes H^l_{a,b}(X',Y))T^2_a(X, X' \otimes T_bY);$$
$$( T_{ab}Y \otimes T^2_a(X,X'))H^r_{a,b}(Y,X \otimes X') = (H^r_{a,b}(X,Y) \otimes T_a X')T^2_a(T_bY\otimes X,X');$$
\item
for $k \labelto{b} j \labelto{a} i$ in $D$ and $X \in\Ob(\C_j)$:
$$(T_a X \otimes T^0_{ab})H^l_{a,b}(X,\un) = T_a(X \otimes T^0_b);$$
$$(T^0_{ab} \otimes T_a X)H^r_{a,b}(\un,X) = T_a(T^0_b \otimes X);$$
\item
for  $k \labelto{b} j \labelto{a} i$ in $D$ and $X \in\Ob(\C_k)$:
$$(T^0_a \otimes T_{ab} X) H^l_{a,b}(\un,X) = {\mu_{a,b}}_X.$$
$$(T_{ab} X \otimes T^0_a) H^r_{a,b}(X,\un) = {\mu_{a,b}}_X.$$
\item for $l \labelto{c} k \labelto{b} j \labelto{a} i$ in $D$ and $X \in \Ob(\C_j), Y \in \Ob(\C_l)$, $Z \in \Ob(\C_k)$, the \emph{pentagon equations}:

$$(H^l_{a,b}(X,Y) \otimes T_{abc}Z) H^l_{a,bc}(X \otimes T_b Y,Z) T(X\otimes H^l_{b,c}(Y,Z))$$
$$\qquad\qquad\qquad\qquad = (T_aX \otimes H^l_{ab,c}(Y,Z) ) H^l_{a,b}(X,Y \otimes T_cZ)$$
$$(T_{abc}Z \otimes H^r_{a,b}(Y,Z))H^r_{a,bc}(Z, T_b Y \otimes X) T(H^r_{b,c}(Z,Y) \otimes X)$$
$$\qquad\qquad\qquad\qquad  = (H^r_{ab,c}(Z,Y) \otimes T_aX)H^r_{a,b}(T_cZ \otimes Y,X)$$

\end{enumerate}
\end{prop}

\begin{proof}
These formulae result from the axioms by straightforward computation.
\end{proof}

\subsection{Hopf polyads}

\begin{defi}
A \emph{left (resp. right) Hopf polyad} is a comonoidal polyad whose left (resp. right) fusion operator is an isomorphism. A \emph{Hopf polyad} is a comonoidal polyad both of whose
fusion operators are invertible.
\end{defi}

\begin{exa}
A Hopf monad can be viewed as a Hopf polyad with source $*$.
\end{exa}

\begin{rem}
An action-type comonoidal polyad which preserves tensor products is a Hopf polyad (by definition of the fusion operators).
\end{rem}

\begin{thm}\label{th-hopf}
If $T$ is an \re-exact left Hopf (resp. right Hopf, resp. Hopf) polyad, so is its lift $\tT$.
\end{thm}

\begin{proof}
Let us prove the left-handed version. Assume $T$ is left Hopf. The left fusion operator of $\tT$:
$$\tilde{H^l}_{a,b}(M,N) : \tT_a(M \otimes \tT_b N) \to \tT_a M \otimes \tT_{ab} N \quad \text{($M$ in  $(\C_k)_{T_k}$ and  $N$ in $(\C_j)_{T_j}$)}$$
is a natural transformation between \re-exact functors. We have the following lemma.

\begin{lem}
Let $\C$ be a category, $T$ a monad on $\C$, and let $A, B : \C_T \to \D$ be two functors preserving reflexive coequalizers. If $\alpha : A \to B$ is a natural transformation
such that for any $c \in \Ob(\C)$, $\alpha_{F_T c}$ is an isomorphism, then $\alpha$ is an isomorphism.
\end{lem}

\begin{proof}
Let $M$ be a $T$-module, with action $\delta : TM \to M$, and let us show that $\alpha_M$ is an isomorphism. We have an exact reflexive cofork in $\C_T$:
$$\xymatrix{F_T T(M) \dar{F_T\delta}{\eps F_T(M)} & F_T(M)  \ar[r]^\delta & M}$$
hence a commutative diagram in $\D$:
$$\xymatrix{
A F_T T(M) \ar[d]_{\alpha_{F_T T(M)}} \dar{AF_T\delta}{A\eps F_T M}
& AF_T(M)  \ar[r]^\delta \ar[d]^{\alpha_{F_T(M)}} & A(M) \ar[d]^{\alpha_M} \\
B F_T T(M) \dar{BF_T\delta}{B\eps F_T M} & BF_T(M)  \ar[r]^\delta & B(M)
}$$
where the horizontal lines are exact. Since the left and center vertical arrows are isomorphisms, so is the right one, that is, $\alpha_M$ is an isomorphism.
\end{proof}

Using the lemma and the fact that $U_{T_i}$ is conservative, it is enough to show that
$U_{T_i} \tilde{H}^l_{a,b}(F_{T_j}(X), F_{T_k}(Y))$ is an isomorphism for $X$ in $\C_j$ and
$Y$ in $\C_k$.

Consider the diagram:
$$\xymatrix@C=0pc{
T_a(T_j X \otimes T_b T_k Y) \ar@{->>}[dddd]_{\text{epi.}} \ar[rrr]^{H^l_{a,b}(T_j X, T_k Y)}_\thickapprox &&& T_aT_jX \otimes T_{ab} T_k Y \ar@{->>}[ddd] \ar[ldd]^{{\mu_{a,j}}_X \otimes {\mu_{ab,k}}_{Y}}\\
&T_a T_j(X \otimes T_b T_k Y) \ar[ul]^{\thickapprox}_{\quad T_a(H^l_{j,b}(X,T_kY))} \ar[rru]_{\quad(\id \otimes \mu_{a,b})(T_a T_j)^2} \ar[d]_{\mu_{a,j}(X \otimes {\mu_{b,k}}_Y)} & 
&\\
&T_a(X \otimes T_b Y)\ar[r]^{H^l_{a,b}(X,Y)}_{\thickapprox} \ar[d]_{\thickapprox} &T_a X \otimes T_{ab} Y \ar[rd]^{\thickapprox}&\\
&U_{T_i} \tT_aF_{T_j}(X \otimes T_b Y) \ar[ld]^{\quad\quad U_{T_i}\tT_a(H^l_{j,b}(X,Y))}_\thickapprox&& U_{T_i}\tT_a F_{T_j} X \otimes U_{T_i}\tT_{ab} F_{T_k}Y\ar[d]^{\thickapprox}\\
U_{T_i}\tT_a(F_{T_j}X \otimes \tT_b F_{T_k}Y ) \ar[rrr]_{U_{T_i} \tilde{H}_{a,b}(F_{T_j}(X), F_{T_k}(Y))} &&& U_{T_i}(\tT_a F_{T_j} X \otimes \tT_{ab} F_{T_k}Y)
}$$
The outer square commutes, and one checks easily that the inner cells commute, except the bottom hexagon. By diagram chasing, one sees that the bottom hexagon commutes also, so the bottom arrow $\tilde{H}^l$ is an isomorphism, that is, $\tT$ is a left Hopf polyad.
\end{proof}

A \re-exact comonoidal polyad $T = (D,\C,T,\mu,\eta)$ is \emph{transitive} if for $j \labelto{a} i$ in $D_1$, $T^0_a(\un)$ is a transitive coalgebra in $\C_i$.

\begin{thm} \label{thm-fond}(Fundamental theorem of Hopf polyads)

If $T$ is a transitive \re-exact left or right Hopf polyad, then $\tT$ is a strong comonoidal action-type polyad - in other words, a strong comonoidal pseudofunctor.
\end{thm}

\begin{proof}

We begin with the following lemma, which is a special case of the theorem (when $\tT = T$).

\begin{lem}\label{lem-case-eta}
Let $T = (D,\C,T,\mu,\eta)$ be a left or right Hopf polyad such that $\eta$ is an isomorphism. Then $T$ preserves tensor products.
If in addition $T$ is transitive, then $T$ is strong comonoidal of action type.
\end{lem}

\begin{proof} Assume for instance that $T$ is left Hopf.
Applying Proposition~\ref{prop-fusion}, Assertion (2), we see immediately that $T^2_a$ is an isomorphism for $a \in D_1$, that is, $T$ preserves tensor products.
In particular, the coproduct $T^2_a(\un)$ of the coagebra $T_a(\un)$ is an isomorphism, which implies that $T^a(\un) \otimes T^0_a$ is an isomorphism.
Now assume $T$ is transitive. The functor $T_a(\un)\otimes ?$ is conservative, so  $T^a_0$ is an isomorphism, that is, $T_a$ preserves tensor units.
Thus, $T_a$ is strong comonoidal.
Moreover, by Proposition~\ref{prop-fusion}, Assertion (6), $\mu_{a,b}(X) =(T^0_a \otimes T_{ab}X) H^l_{a,b}(\un,X)$ is an isomorphism, which shows that $T$
is a polyad of action type, and we are done.
\end{proof}

Now let us prove the general case. We already know that if $T$ is a \re-exact left (resp. right) Hopf polyad , so is $\tT$ (by Proposition~\ref{prop-comopo} and Theorem~\ref{th-hopf}).
If in addition $T$ is transitive, so is $\tT$:

\begin{lem}\label{lem-trans}
If $T$ is a \re-exact transitive left or right Hopf polyad, then $\tT$ is transitive.
\end{lem}

\begin{proof}
Assume $T$ is left Hopf. Let $(j \labelto{a} i) \in D_1$, and consider the  following:
$$\xymatrix{
T_a T_j \un \dar{T_a(T^0_j)}{{\mu_{a,j}}_\un} \ar[d]_{H^l(\un,\un)} & T_a\un \ar[r]^{T^0_a} \ar[d]^{\id_{T_a\un}} & \un \ar[d]^{\id_\un}\\
T_a\un \otimes T_a\un \dar{T_a\un \otimes T^0_a}{T^0_a \otimes T_a \un} & T_a\un \ar[r]^{T^0_a} & \un.
}
$$
The bottom cofork is exact because $T_a\un$ is a transitive coalgebra. The diagram is an isomorphism bewteen the top and bottom coforks (by Proposition~\ref{prop-fusion}, Assertions~(5)
and~(6). So the top cofork is exact, which means that $\tT_a \un$ is isomorphic to $\un$, and therefore transitive as a coalgebra.
\end{proof}

We conclude by applying Lemma~\ref{lem-case-eta} to the polyad $\tT$.
\end{proof}

\begin{rem}
If $T$ is a \re-exact left or right Hopf monad, then $\tT$ preserves tensor products. (Results from the previous proof).
\end{rem}

\section{Fonctoriality properties of polyads}\label{sect-fonct}

\subsection{Morphisms of polyads}

Let $D$ be a small category. We define the \emph{$2$-category of $D$-polyads} $\Poly(D)$ as follows.

An object of $\Poly(D^)$ is a $D$-polyad.

Given two $D$-polyads $T=(D,\C,T,\mu,\eta)$ and $T' = (D,\C',T',\mu',\eta')$, a $1$-morphism from $T$ to $T'$ in $\Poly(D)$ consists in a pair $(F,\theta)$ as follows:
\begin{enumerate}[(1)]
\item $F = (F_i)_{i \in D_0}$ is a family of functors $F_i : \C_i \to \D_i$;
\item $\theta = (\theta_a)_{a \in D_1}$ is a family of natural transformations $\theta_a : F_i T_a \to T'_a F_j$, where $j \labelto{a} i$;
\end{enumerate}
with the following conditions:
\begin{enumerate}
\item for $k \labelto{b} j \labelto{a} i$ in $D_2$,
$$\theta_{ab} (F_i \mu_{a,b}) = (\mu'_{a,b}F_k)(T'_a \theta_b)(\theta_a T_b),$$
\item  for $i \in \D_0$,
$$\theta_i(F_i \eta_i) = \eta'_i F_i.$$
\end{enumerate}

Given two $1$-morphisms $(F,\theta)$, $(F',\theta')$ from $T$ to $T'$, a $2$-morphism $\alpha : (F,\theta) \to (F',\theta')$ is a collection $(\alpha_i)_{i \in D_0}$ of natural transformations $\alpha_i : F_i \to F'_i$ such that for any $j \labelto{a} i$ in $D_1$,
$$\theta'_a (\alpha_i T_a) = (T'_a \alpha_j) \theta_a.$$

The various compositions are defined in the obvious way. In particular, the composition of two $1$-morphisms $(F,\theta) : T \to T'$ and $(F',\theta'): T' \to T''$ is the $1$-morphism $(F'F,\theta'')$, where
$$\theta''_a = (\theta'_a F_j) (F'_i \theta_a).$$

We define similarly the \emph{2-category of comonoidal $D$-polyads} $\coPoly(D)$, whose objects are comonoidal $D$-polyads, $1$-morphisms are $1$-morphisms of $D$-polyads $(F,\alpha)$
equipped with a comonoidal structure on the functors $F_i$ such that the natural transformations $\alpha_a$ are comonoidal, and $2$-morphisms are $2$-morphisms of $D$-polyads $\alpha$
such as the natural transformations $\alpha_i$ are comonoidal.

\begin{exa}
Let $T = (D,\C,T,\mu,\eta)$ be a \re-exact $D$-polyad, with lift $\tT = (D,\tilde{\C},\tT, \tilde{\mu},\tilde{\eta})$.
Then we have two $1$-morphisms of $D$-polyads $U_T = (U_T,e) :  \tT \to T$ and $F_T = (F_T,h) : T \to \tT$, defined as follows:
\begin{align*}
&(U_T)_i = U_{T_i}, &h_a = U_{T_i} \tT_a \eps^{T_j} : U_{T_i} \tT_a    \to U_{T_i} \tT_a F_{T_j}  U_{T_j} \simeq T_a U_{T_j},\\
&(F_T)_i = F_{T_i}, &e_a = \eta^{T_i} \tT_a F_{T_j} : F_{T_i} T_a \simeq  F_{T_i} U_{T_i} \tT_a F_{T_j} \to  \tT_a U_{T_j}.
\end{align*}
Moreover, we have $2$-morphisms $\eta = (\eta_i) : \id_T \to U_T F_T$ and $\eps = (\eps_i) : F_T U_T \to \id_{\tT}$, so that $(F_T,U_T,\eta,\eps)$ is an adjunction in the $2$-category $\Poly(D)$.
Moreover, if $T$ is a \re-exact comonoidal $D$-polyad, the above construction can be made comonoidal, and produces an adjunction $(F_T,U_T,\eta,\eps)$ in $\coPoly(D)$.
\end{exa}

\subsection{Pull-back}

Let $T =(D,\C,T,\mu,\eta)$ be a $D$-polyad, and consider a functor $f : D' \to D$. Then one defines a $D'$-polyad $f^*T = (D',\C',T',\mu',\eta')$, with $\C' = (\C_{f(j)})_{j \in D'_0}$,
$T'_a = T_{f(a)}$ for $a \in D'_1$, $\mu'_{a,b} = \mu_{f(a),f(b)}$ for $(a,b) \in D'_2$, and $\eta'_j = \eta_{f(j)}$ for $j \in D'_0$.

If $T$ is comonoidal, then $f^*$ becomes comonoidal too.
This construction extends naturally to $2$-functors \emph{pull-back along $f$}:
$$f^* : \Poly(D) \to \Poly(D') \blabla{and} f^* : \coPoly(D) \to \coPoly(D').$$

Note that the pull-back preserves left and right Hopf polyads, and it commutes with the lift: if $T$ is a \re-exact (comonoidal) polyad, then so is $f^* T$, and $f^*\tilde{T} = \widetilde{f^* T}$.

Given a $D$-polyad $T$ and a functor $f : D' \to D$, we define a functor: $$f_* : \Mod(T) \to \Mod(f^*T),$$
by sending a $T$-module $(X,\rho)$ to the $f^*$-module $(X',\rho')$,
with $X' = (X_f(j))_{j \in D'_0}$, $\rho' = (\rho_{f(a)})_{a \in D'_1}$.

If $T$ is a comonoidal $D$-polyad, then $f_*: \Mod(T) \to \Mod(f^*X)$ is strong (in fact, strict) comonoidal.

In the case where $f$ is a inclusion of a subcategory $D'$ in $D$, then the pull-back $f^* T$ will be denoted by $T_{\mid D'}$.

Using the notion of pull-back, we can interpret representations of a polyad as modules of a certain pull-back of the same polyad:

\begin{lem}\label{lem-mod-rep}
Let $D$ be a category, and define a new category $\Ar(D)$ in the following way. The objects of $\Ar(D)$ are the morphisms of $D$: $\Ar(D)_0 = D_1$. Given $b,c \in \Ar(D)$, a morphism
$a : b \to c$ is a morphism $a$ such that $c = ab$. Note that the map $(a,b) \to a : b \to ab$ is a bijection between $D_2$ and $\Ar(D)_1$.
Define a functor
$$\left\{
\begin{array}{llll}
 t: & \Ar(D) &\to &D  \\
  & \Ar(D)_0 \ni b &\mapsto &t(b)\\
  &  \Ar(D)_1 \ni a &\mapsto &a
\end{array}
\right.$$
Then, if $T$ is a polyad with source $D$,  the categories $\Rep(T)$ and $\Mod(t^* T)$ are canonically isomorphic, and if $T$ is a comonoidal polyad, this isomorphism is strict (co)monoidal.
\end{lem}

\begin{proof}
Straightforward.
\end{proof}

\subsection{Special case:  when the source is a groupoid}

\begin{thm} Let $\Gamma$ be a connected groupoid, and let $T$ be a left or right \re-exact transitive Hopf $\Gamma$-polyad. Let $i_0\in \Gamma$, and
denote by $\Gamma_{i_0}$ the full subcategory of $\Gamma$ with $i_0$ as unique object.
Then:
\begin{enumerate}[(1)]
\item the restriction functor $\Mod(T) \to \Mod(T_{\mid \Gamma_{i_0}})$ is a monoidal equivalence;
\item the group $G= \Aut_\Gamma(i_0)$ acts on the monoidal category $\tilde{\C}_{i_0} = (\C_{i_0})_{T_{i_0}}$;
\item $\Mod(T)$ is monoidally equivalent to the equivariantization of $\tilde{\C}_{i_0}$ under this action.
\end{enumerate}
\end{thm}

\begin{proof} In order to simplify the notation, set  $\C_0 = \C_{i_0}$ and $T_0 = T_{\mid \Gamma_{i_0}}$.
In view of Proposition~\ref{prop-comopo} and the previous remarks, we may replace $T$ with $\tT$, and by Theorem~\ref{thm-fond}, we may assume that $T$ is of action type and strong comonoidal. In particular, for $a \in D_1$ $T_a$ is strong comonoidal, and it is an equivalence of categories because $a$ is invertible. Recall also that in this setting, actions are invertible (see Lemma~\ref{lem-act-inv}).

In that case, $T_0$ is nothing but an action of the group $G$ on $\C_0$, and $\Mod(T_0)$ is the equivariantization of this action.
So all that remains to be proved, is Assertion~(1).

We construct a quasi-inverse to the restriction functor as follows. Choose for each $i \in D_0$ a morphism $a_i : i_0 \iso i$ in $D_1$, which is possible because $\Gamma$ is connected.

Let $(X_0,r)$ be an object of $\Mod(T_0)$. Here $X_0$ is an object of $\C_0$ and $r = (r_g)_{g \in G}$, with $r_g : T_g(X) \iso X$.
Define an object $(X,\rho)$ in $\Mod(T)$ as follows. For $i \in D_1$, set $X_i = T_{a_i} X_0$. For $j \labelto{a} i$ in $D_1$, define $\rho_a : T_a X_j \to X_i$ to be the compositum of
$$T_a X_j = T_aT_{a_j} X_0 \simeq T_{a_i} T_{\gamma} X_0\labelto{T_{a_i} r_{\gamma}} T_{a_i} X_0 = X_i,$$
where $\gamma = {a_i}^{-1} a a_j \in G$.

One verifies that $(X,\rho)$ is a $T$-module, and the assignment $(X_0,r) \mapsto (X,\rho)$ defines a quasi-inverse to the restriction functor $\Mod(T) \to \Mod(T_0)$.

\end{proof}

\section{Quasitriangular comonoidal and Hopf polyads}\label{sect-quasitriangular}

If $T$ is a comonoidal monad on a monoidal category $\C$, an \Rt matrix for $T$ is a natural transformation is (see~\cite{BV2}) a natural transformation
$$R_{X,Y} : X \otimes Y \to Y \otimes X$$ satisfying the following axioms:

\begin{thm}\label{thm-rmatbraid}
A \Rt matrix for $T$ defines a braiding $\tau$ on the monoidal category $\C_T$ as follows:
$$\tau_{(X,r),(Y,s)} = (s \otimes t) R_{X,Y} : (X,r) \otimes (Y,s) \to (Y,s) \otimes (X,r),$$
and this assignment is a bijection between \Rt-matrices for $T$ and braidings on $\C_T$.
\end{thm}

Let $T = (D,\C,T,\mu,\eta)$ be a comonoidal polyad. An \Rt matrix for $T$ is a family $(R^i)_{i \in D_0}$, where $R_i$ is a \Rt matrix for $T_i$ for each $i \in D_0$, satisfying the following axiom for each $j \labelto{a} i$ in $D_1$:
$$
({\mu_{a,j}}_Y \otimes {\mu_{a,j}}_X) \,  T^2_ a(T_jY \otimes T_j X)\,T_a(R^j_{X,Y})= ({\mu_{i,a}}_Y \otimes {\mu_{i,a}}_X)\,R^i_{T_aX,T_aY}\,T^2_a(X \otimes Y)
$$

\begin{rem}
Note that this set of axioms is redundant, while the last axiom implies the first axiom satisfied by $R_i$ as  \Rt matrix (when $a=\id_i)$.
\end{rem}

\begin{thm}\label{thm-polyrmatbraid} Let $R$ be a \Rt matrix on a comonoidal polyad $T = (D,\C,T,\mu,\eta)$, and denote by $\tau_i$ the braiding on $(\C_i)_{T_i}$ defined by the
 \Rt matrix $R_i$ for $i \in T_0$.
Then, for every functor $f : D' \to D$, $R$ defines a braiding $\tau^f$ on the category $\Mod(f*T)$,
defined as follows:
$$\tau^{f}_{(X, \rho),(Y,\sigma)} = \bigl(\tau^{f(x)}_{X_i,Y_i}\bigr)_{x \in D'_0},$$
and, given a functor $g: D'' \to D'$, with the braidings so defined, the restriction functor $\Mod(f^*T) \to \Mod(g^*f^*T)$ is braided.

This assignment defines a bijection between \Rt-matrices for $T$ and ways of choosing a braiding on each categegory $\Mod(f^*)$, $f : D' \to D$, so that the restriction functors
$\Mod(g^*f^*T) \to \Mod(f^*T)$ are braided for $D'' labelto{g} D' \labelto{f} D$.
\end{thm}

\begin{proof}
We verify first that $\tau = \tau^{\id_D}$ is a braiding on $\Mod(T)$.
We check first that, given $(X,\rho)$ and $(Y,\sigma)$ in $\Mod(T)$, $(\tau^i_{X_i,Y_i)}$ is a morphism in $\Mod(T)$ from $(X,\tau) \otimes (Y,\sigma)$ to $(Y,\sigma) \otimes (X,\rho)$.
This means that, for $j \labelto{i} i$ in $D_1$,
$$(\sigma_a \otimes \rho_a)\, T^2_a(Y_j \otimes X_j)\, T_a(\tau^j_{X_j,Y_j}) = \tau^i_{X_i,Y_i} \,(\rho_a \otimes \sigma_a)\,T^2_a(X_j \otimes Y_j),$$
which results from the following computation:
\begin{align*}
(\sigma_a &\otimes \rho_a)\, T^2_a(Y_j \otimes X_j)\, T_a(\tau^j_{X_j,Y_j}) = (\sigma_a \otimes \rho_a) T^2_a(Y_j \otimes X_j) T_a(\sigma_j \otimes \rho_j)T_a(R_{X_j,Y_j})\\
&=(\sigma_a T_a\sigma_j\otimes \rho_a T_a\rho_j) T^2_a(T_jY_j \otimes T_jX_j)T_a(R_{X_j,Y_j}) \\
&= (\sigma_a {\mu_{a,j}}_{Y_j}\otimes \rho_a {\mu_{a,j}}_{X_j}) T^2_a(T_jY_j \otimes T_jX_j)T_a(R_{X_j,Y_j})\\
&= (\sigma_a  \otimes \rho_a)({\mu_{a,j}}_{Y_j}\otimes {\mu_{a,j}}_{X_j}) T^2_a(T_jY_j \otimes T_jX_j) T_a(R_{X_j,Y_j})\\
&= (\sigma_a  \otimes \rho_a) ({\mu_{i,a}}_{Y_j} \otimes {\mu_{i,a}}_{X_j})\,R^i_{T_aX_j,T_aY_j}\,T^2_a(X_j \otimes Y_j)\\
&= (\sigma_i {\mu_{i,a}}_{Y_j} \otimes \rho_i {\mu_{i,a}}_{X_j})\,R_{T_a X_j,T_a Y_j}\,T^2_a(X_j \otimes Y_j)\\
&= (\sigma_i T_a\sigma_a \otimes \rho_i T_a\rho_a)\,R_{T_a X_j,T_a Y_j}\,T^2_a(X_j \otimes Y_j)\\
&=(\sigma_i \otimes \rho_i) \,R_{X_i,Y_i} \,(\rho_a \otimes \sigma_a)\,T^2_a(X_j \otimes Y_j)
=\tau^i_{X_i,Y_i} \,(\rho_a \otimes \sigma_a)\,T^2_a(X_j \otimes Y_j).
\end{align*}
One shows easily that $\tau$ is a natural transformation, and it is a braiding because each $\tau_i$ is a braiding and the forgetful functor $U_T : \Mod(T) \to |T|$ is conservative.

Now  given a functor  $f : D' \to D$,  the \Rt matrix $R$  for $T$ defines by restriction a \Rt matrix $f^*R = (R_i)_{i \in D'_0}$ for $f^*T$, which defines therefore a braiding
on $\Mod(f^*T)$. Clearly if $g : D'' \to D'$ is a second functor, the strong comonoidal restriction functor $\Mod(f^*T) \to \Mod(g^*f^*T)$ induced by $g$ preserves braidings.

The assigment $R \mapsto (\tau^f)_{f : D' \to D}$ is clearly intro, since the braidings $\tau_i$, $i \in D_0$, determine the \Rt-matrix $R$.
Let us check that it is onto.

Consider a data $(\tau^f)_{f : D' \to D}$, where $\tau^f$ is a braiding $\tau^f$ on $\Mod^f(T)$, and the restriction functors $\Mod(f^*T) \to \Mod(g^*f^*T)$
(for $D'' \labelto{g} D' \labelto{f} D$) are braided. In particular, for $i \in D_0$, let $\underline{i} : * \to D$ be the functor which sends the unique object of $*$ to $i$.
We have $\underline{i}^* T = T_i$, hence a braiding $\tau^{\underline{i}}$ on $(\C_i)_{T_i}$, which corresponds with an \Rt matrix $R_i$ for $T_i$ according to Theorem~\ref{thm-rmatbraid}.
We verify that $R=(R_i)$ is a \Rt matrix for $T$. All we have to check is that for each $j \labelto{a} i$ in $D_1$:
$$
({\mu_{a,j}}_Y \otimes {\mu_{a,j}}_X) \,  T^2_ a(T_jY \otimes T_j X)\,T_a(R^j_{X,Y})= ({\mu_{i,a}}_Y \otimes {\mu_{i,a}}_X)\,R^i_{T_aX,T_aY}\,T^2_a(X \otimes Y)
$$
Now let $\underline{a}$ be the functor $\Delta_1 \to D$ which sends the arrow $0 \to 1$ to $a$. The category $\Mod(\underline{a}^*T)$ is braided, with  braiding $\tau^{\underline{a}}$.
If $X$ is an object of $\C_j$, define an object $F_X$ of $\Mod(\underline{a}^*T)$ by $F_X = ((T_j X, T_a X),\rho^X)$, which $\rho_0 = {\mu_{i,i}}_X$, $\rho_1 = {\mu{i,a}}_X$,
$\rho_{0 \to 1} = \mu_{a,j}$. When we express the fact that, for $X, Y \in \Ob(\C_j)$, $\tau^{\underline{a}}_{F_X,F_Y}$ is a morphism in $\Mod(\underline{a}^*T)$, we obtain the relation
we had to prove.

By construction, the \Rt matrix $R$ produces the collection of braidings $(\tau^f)$, so we are done.
\end{proof}

See Remark~\ref{rem-ex-centralizer} for examples.

\section{Hopf polyads vs Hopf monads}\label{sect-polyads-monads}

We have noted that, if $T$ is a polyad, the forgetful functor $\Mod(T) \to |T|$ doesn't have a left adjoint in general. In this Section, we consider  polyads for which
such an adjoint exist, and compare them with monads.


Let $T = (D,\C,T,\mu,\eta)$ be a polyad, and for $i \in D$, let $N^D_i$ be the number of arrows in $D$ whose target is $i$.
We say that $T$ is \emph{\su-exact} if for each $i \in D_0$, the category $\C_i$ admits finite sums, and (should $N^D_i$ be infinite) sums of up to $N^D_i$ objects, and for
$a: i \to j$ in $D$, the functor $T_a$ preserves such sums.

A \emph{\su-exact comonoidal polyad} is a comonoidal polyad  $T = (D,\C,T,\mu,\eta)$ which is  \su-exact as a polyad, and in addition,  satisfies that for each $i \in D_0$ the tensor product of $\C_i$ preserves finite sums and sums of $N^D_i$ objects in each variable. A \emph{Hopf (resp. left Hopf, left right Hopf) \su-exact polyad} is a comonoidal \su-exact polyad which is Hopf (resp. left Hopf, resp. right Hopf).

If $T$ is \su-exact, set for $X \in \Ob(|T|)$ $$\hT \, X = \bigl( \coprod_{k \labelto{b} i \in D_1} T_b X_k \bigr)_{i \in D_0}.$$

For $a \in D_1$, $a : j \to i$, the product $\mu$ defines a morphism $$\rho_a : T_a (\hT \,X)_i \simeq \coprod_{k \labelto{b} i} T_a T_b X_k \to \coprod_{k \labelto{c} j} T_c X_k =
(\hT\, X)_j.$$

The product $\mu$ and the unit $\eta$ define also morphisms
$$\hat{\mu}_X : \hT^2 \,X  \simeq (\coprod_{k \labelto{b} j \labelto{a} i} T_a T_b X_k)_{i \in D_0} \to (\coprod_{k \labelto{c} i}  T_c X_k)_{i \in D_0} = \hT\, X$$
$$\hat{\eta}_X : X \to (\coprod_{j \labelto{a} i} T_a X_j)_{i \in D_0}.$$

\begin{thm}\label{thm-poly-mono}
Let $T$ be a \su-exact polyad. Then:
\begin{enumerate}[(1)]
\item The forgetful functor $U_T : \Mod(T) \to |T|$ has a left adjoint, defined by
$$F_T(X) = (\hT\,X,\rho),$$
and the adjunction $(F_T \vdash U_T)$ is monadic, with monad $(\hT,\hat{\mu},\hat{\eta})$; in fact $\Mod(T)$ is isomorphic to the category $|T|_\hT$ of $\hT$-modules $|T|$.
\item If $T$ is a  comonoidal polyad, then $\hT$ is a comonoidal monad and $\Mod(T) \simeq |T|_\hT$ as monoidal categories.
\end{enumerate}
\end{thm}

\begin{proof}
Let us prove Assertion (1). One verifies easily that $(\hT,\hat{\mu},\hat{\eta})$ is a monad. Now consider an object $X$ in $|T|$.
A morphism $r : \hT X \to X$ in $|T|$ is a family of morphisms $(r_i)_{i \in D_0}$, with $r_i : \coprod_{j \labelto{a} i} T_a X_j \to X_i$, and, in view of the universal property of sums,
it can be viewed also as a family $\overline{r} =(\overline{r}_a : T_a X_j \to X_i)_{(j \labelto{a} i) \in D_1}$.

Moreover, we have the following equivalences:
$$r(Tr) = r \mu_X \iff \forall (k \labelto{b}j \labelto{a} i) \in D_1, \overline{r}_a (T_a \overline{r}_b) = \overline{r}_{ab} {\mu_{a,b}}_{X_k},$$
$$r \eta_X = \id_X \iff \forall i \in D_0, \overline{r}_i \eta_i  = \id_{X_i}.$$
As a result, $r$ is an action of the monad $\hT$ on $X$ if and only if $\overline{r}$ is an action of the polyad $T$ on $X$. Hence an isomorphism of categories:
$$\left\{
\begin{array}{llll}
K:&\Mod(T) &\to &|T|_{\hT} \\
&(X,\overline{r}) &\mapsto &(X,r).
\end{array}
\right.$$
In particular, the forgetful functor $U_T$ is monadic, with left adjoint $K F_{\hT} = F_T$.

Now for Assertion (2). It is a general fact that the left adjoint of a strong comonoidal functor admits a unique comonoidal structure such that the adjunction is comonoidal. Hence a comonoidal structure on the monad of the adjunction.
One verifies easily that the comonoidal structure on $\hT$ so defined coincides with the one we decribed previously. The rest ensues.
\end{proof}

However, $\hT$ need not be a Hopf monad if $T$ is a Hopf polyad.

\begin{exa}
Consider the constant polyad $T = \V_D$ of Example~\ref{exa-constant}. If $\V$ is a monoidal category, then $T$ is in fact a Hopf polyad. Consider the special case $\V=\vect_\kk$ for some field $\kk$, and $D = \Delta_1$. In that case the monad $\hT$ on $|\V| = \vect_\kk \times \vect_\kk$ is given by
$\hT(E_0,E_1) = (E_0,E_0 \oplus E_1)$. Therefore  $$\hT(E,\hT F) = (E_0 \otimes F_0, E_0 \otimes F_0 \oplus E_1 \otimes (F_0 \oplus F_1) ).$$ On the other hand,
$$\hT(E) \otimes \hT(F) =(E_0 \otimes F_0, (E_0 \oplus E_1) \otimes (F_0 \oplus F_1)).$$  As a result the left fusion (and by reason of symmetry, the right fusion) operator of $\hT$ cannot be an isomorphism, that is, $\hT$ is not a Hopf monad.
\end{exa}

We have however the following theorem.

\begin{thm}\label{thm-hopfpolyhopfmon}
Let $T$ be a \su-exact comonoidal polyad with source $D$.
If  $T$ is a left (resp. right) Hopf polyad and $D$ is a groupoid,
then  $\hT$ is a left (resp. right) Hopf monad.
\end{thm}

\begin{proof}
Let us tackle the left-handed version.
Let $X,Y$ be objects of $|T|$. For each $i \in D_0$ we have a commutative square in $\C_i$:
$$\xymatrix{
\hT(X \otimes \hT(Y))_i \ar[d]_{\simeq}\ar[rrrr]^{(H^l_{X,Y})_i} &&&&\hT(X)_i\otimes \hT(Y)_i\ar[d]^{\simeq} \\
\coprod \limits_{k \labelto{b} j \labelto{a} i}\!\!\!\! T_a(X_j  \otimes T_b Y_k) \ar[rrr]_{\coprod\limits_{k \labelto{b} j \labelto{a} i} H^l_{a,b}(X_j,Y_k)}&&& \coprod\limits_{k \labelto{b} j \labelto{a} i} \!\!\!T_a(X_j)  \otimes T_{ab} X_k \ar[r]_\iota & \coprod\limits_{m \labelto{c} i\atop n \labelto{d} i}\!\!\! T_c X_m  \otimes T_d Y_n,
}$$
where $\hat{H}^l$ denotes the left fusion operator of the comonoidal monad $\hT$, and $\iota$ is induced by the map $I : (a,b) \mapsto (c,d)= (a,ab)$.
If $D$ is a groupoid, then $I$ is bijective and $\iota$ is an isomorphism. If in addition $T$ is left Hopf, then all $H^l_{a,b}$'s are isomorphisms, so
$\hat{H}^l$ is an isomorphism. Therefore, $\hT$ is left Hopf.
\end{proof}

As a result, theorems about Hopf monads can sometimes be applied to \su-exact polyads whose source is a groupoid. We will see an example of this in Section~\ref{sect-hopfmod}.

\pagebreak[5]

\section{Hopf modules and Hopf representations for Hopf polyads}\label{sect-hopfmodrep}

We have seen that hhe notion of modules over a monad can be generalized in two different ways to polyads: modules and representations. Consequently, Hopf modules can also be generalized in two different ways, namely Hopf modules and Hopf representations

\goodbreak

\subsection{Hopf modules for Hopf polyads}\label{sect-hopfmod}

We exploit the relation between Hopf monads and \su-exact Hopf polyads established in Section~\ref{sect-polyads-monads}  to define Hopf modules and formulate a decomposition theorem.

Let $T$ be a \su-exact comonoidal polyad.
A \emph{(left) $T$-Hopf module} consists in the following data:
\begin{enumerate}[(1)]
\item a $T$-module $(X,\rho)$,
\item a collection $\delta = (\delta_i)_{i \in D_0}$, where $\delta_i : X_i \to (\coprod_{a \in D_1} T_a \un)  \otimes X_i$ is a left coaction of the
coalgebra $\coprod_{a\in D_1} T_a(\un)$ on $X_i$,
\end{enumerate}
such that for $k \labelto{b} j \labelto{a} i$ in $D_2$, the following diagram commutes:
$$\xymatrix{
T_a X_j\ar[dd]_{\rho_a}\ar[r]^{T_a\delta_{j} \qquad}
    &  T_a(\coprod\limits_{k \labelto{b} j} T_b \un\otimes X_j) \ar[r]^\simeq
    &\ar[d]^{\coprod\limits_{b}(T_{ab}\un \otimes \rho_b) T^2_a (T_b\un,X_j)} \coprod\limits_{b} T_a( T_b\un \otimes X_j)\\
&&\coprod\limits_{b} T_{ab}\un \otimes X_i\ar[d]^{\text{can.}}\\
X_i  \ar[r]_{\delta_{i}\qquad}
    &  \coprod\limits_{k \labelto{c} i}T_{c} \un \otimes  X_i \ar[r]_\simeq
    & \coprod\limits_{k \labelto{a}i} T_{c}\otimes X_i\un.
}$$

A \emph{morphism of left $T$-Hopf modules} $f$ between two left $T$-Hopf modules $(X,\rho,\delta)$ and $(X',\rho',\delta')$ is a morphism of $T$-modules $f : (X,\rho) \to (X',\rho')$ such that for $i \in \D_0$ and
$a : j \to i$, $f_i$ is a morphism of comodules $(X_i,\delta_i) \to (X'_i,\delta'_i)$.

Thus, left $T$-Hopf modules form a category $\HM^l(T)$.

The \emph{coinvariant part} of a $T$-Hopf module $X=(X,\rho,\delta)$ is, if it exists, the object of $|T|$ defined by
$$X^{\coinv} = \bigl(\eq( \xymatrix{X_i \dar{\delta_{i}}{ {\eta_i}_\un\otimes X_i} & (\coprod\limits_{j \labelto{a}i} T_a\un)\otimes X_i} )\bigr)_{i \in D_0}.$$
The parallel pairs involved are coreflexive (with common retraction $ \coprod T^0_a\otimes X_i$), in particular the coinvariant
part is always defined if the categories $\C_i$ have coreflexive equalizers.  In that case, we dispose of a functor
$$\left\{
\begin{array}{llll}
?^\coinv :  &\HM^l(T) &\to &|T| \\
&  X &\mapsto &M^\coinv.
\end{array}
\right.$$

A \emph{weakly conservative polyad} is a polyad $T$ satisfying: for each $i \in D_0$, a morphism $f \in \C_i$ such that $T_a(f)$ is an isomorphism for each $(i \labelto{a} j) \in D_1$,
is an isomorphism.

A \su-exact polyad  $(D,\C,T,\mu,\eta)$ \emph{has conservative sums} if every for $i \in D_0$, the sum is conservative in $\C_i$ (that is, if $f$, $g$ are morphisms in $\C_i$ such that
$f \coprod g$ is an isomorphism, then $f$ and $g$ are isomorphims).

A \core-exact \su-exact  polyad is a \su-exact polyad $(D,\C,T,\mu,\eta)$
such that for $i \in D_i$, $\C_i$ admits equalizers of coreflexive pairs,
the tensor product of $\C_i$ preserves them,
finite sums and sums of at most $N^D_i$ terms preserve them, and
for $j \labelto{a} i$ in $D_1$, $T_a : \C_j \to \C_i$ preserves them.

\begin{thm}
Let $T$ be a  a \core-exact  \su-exact weakly conservative left Hopf polyad having conservative sums and whose source is a groupoid. Then the coinvariant part functor $\coinv : \HM(T) \to |T|$ is an equivalence of categories.
\end{thm}

\begin{proof} Let $T=(D,\C,T,\mu,\eta)$ be a left Hopf polyad satisfying the conditions of the theorem, and let $\hT$ be the comonoidal monad associated with $T$ according to Theorem~\ref{thm-poly-mono}; it is in fact a left Hopf monad by Theorem~\ref{thm-hopfpolyhopfmon}.
A $T$-Hopf module $(X,\rho,\delta)$ can be viewed as a $\hT$-Hopf module. Indeed, the $T$-module $(X,\rho)$ can be viewed as a $\hT$-module according to Assertion (1) of Theorem~\ref{thm-poly-mono};
the data $\delta$ is a left coaction of the comonad $\hT \un$ on $X$, and the axiom expressing the compatibility of $\rho$ and $\delta$ translates into the compatibility of the action and coaction of a left $\hT$-Hopf module. Thus we have an equivalence of categories between $\HM^l(T)$ and the category $\HM^l(\hT)$ of left $\hT$-Hopf modules, which preserves coinvariant parts by construction.

Moreover, our exactness assumptions ensure that $|T|$ has coreflexive equalizers and $\hT$ preserves them. Lastly, $\hT$ is
conservative. Indeed, let $f$ be a morphism in $|T|$. If $\hT(f)$ is an isomorphism, then for every $i \in D_0$, $\coprod_{j\labelto{a}i} T_a(f_i)$ is an isomorphism, and since sums are conservative,
this implies that each $T_a(f_i)$ is an isomorphism. Since $T$ is weakly conservative, this implies that $f$ is an isomorphim.
We may therefore apply to $\hT$ Theorem~6.11. of \cite{BLV}, and conclude that $\HM^r(\hT)$ is equivalent to $|T|$ via the functor equivariant part, hence $\coinv :\HM^l(T) \to |T|$ is an equivalence.
\end{proof}

\subsection{Hopf representations for Hopf polyads}
We define Hopf representations of a Hopf polyad, and formulate a decomposition theorem.

Let $T = (D,\C,T,\mu,\eta)$ be a comonoidal polyad. The category $\Rep(T)$ of representations of $T$ (see Section~\ref{sect-rep-poly}) is monoidal, with tensor product defined by
$$(W,\rho) \otimes (W',\rho') = (W \otimes W', \rho''), \blabla{where} \rho''_{a,b} = (\rho_a \otimes \rho_b) T^2_a(W_a,W_b),$$
and unit object  $\un = ((\un)_{a \in D_1},(T^0_a)_{(a,b) \in D_1})$.

The forgetful functor $V_T : \Rep(T) \to |T|$ is strict (co)monoidal, and its left adjoint $L_T$ is comonoidal.

A \emph{left Hopf representation of $T$} is a data $(W,\rho,\delta)$, where
\begin{enumerate}[(1)]
\item $W = (W_a)_{a \in D_1}$, with $A_a \in \Ob(\C_j)$ for $j \labelto{a} i$ in $\D_1$,
\item $\rho= (\rho_{a,b})_{(a,b) \in D_2}$, where $\rho_{a,b}$ is a morphism $T_a W_b \to W_{ab}$ for $k \labelto{b} j \labelto{a} i$ in $D_2$,
\item $\delta = (\delta_a)_{a \in D_1}$, where $\delta_a : W_a \to T_a\un\otimes W_a$,
\end{enumerate}
with the following axioms:

\begin{enumerate}[(i)]
\item $(W,\rho)$ is a Hopf representation of $T$ (see Section~\ref{sect-rep-poly});
\item for $a \in D_1$, $(W_a,\delta_a)$ is a left comodule of the coalgebra $T_a\un$;
\item for $k \labelto{b} j \labelto{a} i$ in $D_2$, the following diagram commutes:
$$
\xymatrix@C=6em{
T_aW_b \ar[d]_{T_a\delta_b}\ar[rr]^{T_a\delta_b}& & W_{ab}\ar[d]^{\delta_{ab}}\\
T_a(T_b\un\otimes W_b)\ar[r]_{T^2_a(T_b\un,W_b)} & T_aT_b\un \otimes T_aW_a \ar[r]_{ {\mu_{a,b}}_\un\otimes\rho_{a,b}}& T_{ab}\un\otimes W_{ab}.
}
$$
\end{enumerate}
\begin{rem}
Since $L_T : |T| \to \Rep(T)$ is comonoidal, $L_T(\un)$ is a coalgebra in $\Rep(T)$, and conditions (i), (ii), (iii) mean that $(W,\rho)$ is an object of $\Rep(T)$ and
$\delta$ is a left coaction of $L_T(\un)$ on $(W,\rho)$.
\end{rem}

Left Hopf representations of $T$ form a category $\HRep^l(T)$, the morphisms between two Hopf representations being representation morphisms which are at the same time comodule morphisms.

If $X$ is an object of $|T|$,  we define a right Hopf representation $\mathfrak{h}^l$ by setting:
$$\mathfrak{h}^l(X) = ((T_a X_{s(a)})_{a \in D_1},({\mu_{a,b}}_{s(b)})_{(a,b) \in D_2},(T^2_a(\un, X_{s(a)}))_{a \in D_1}).$$
Conversely, observe that if $W = (W,\rho,\delta)$ is a left Hopf representation of $T$, then for each $i \in D_0$, $(W_i, \rho_{i,i},\delta_i)$ is a left Hopf module of the bimonad $T_i$.

We say that $(W,\rho,\delta)$ \emph{has a coinvariant part} if for each $i$,  $(W_i, \rho_{i,i},\delta_i)$ has a coinvariant part $W^\coinv_i$ - that is, an equalizer of the parallel pair:
$$W^\coinv_i = \eq(\xymatrix{W_i \dar{ {\eta_i}_\un \otimes W_i }{\delta_i} & T_i\un\otimes W_i}).$$
If such is the case, the object $W^\coinv = ({W_i}^\coinv)_{i \in D_0}$ of $|T|$ is called the \emph{coinvariant part of $W$}.
We say that $T$ \emph{preserves coinvariant parts} if for every $j \labelto{a} i$ in $D_1$, $T_a : \C_j \to \C_i$ preserves the equalizer $W_j$.

We say that a polyad $T$ with source $D$ is \emph{conservative} if for every $a \in D_1$, the functor $T_a$ is conservative.

\begin{thm}\label{thm-str-hopfrep}
Let $T$ be a left Hopf polyad. Assume that the  left Hopf representations of $T$ have coinvariant parts and $T$ preserves them.
Then the following assertions are equivalent:
\begin{enumerate}[(i)]
\item The functor $$\mathfrak{h}^l\co \left\{ \begin{array}{ccl} |T| & \to & \HRep^l(T) \\ X & \mapsto &\mathfrak{h}^l(X) \end{array} \right. $$ is an equivalence of categories;
\item $T$ is conservative,
\end{enumerate}
If these hold, the functor `coinvariant part' $?^\coinv : \HRep^l(T) \to |T|$ is quasi-inverse to $\mathfrak{h}^l$.
\end{thm}

\begin{proof}
Note that Theorem~\ref{thm-str-hopfrep} is true for $D=*$: in that case, it results immediately from Theorem~6.11. of \cite{BLV}, that is, the decomposition theorem for the left Hopf modules of a right Hopf monad.

Let us fist show (i) $\implies$ (ii). Assume that $?^H$ is an equivalence. In particular, it is conservative. On the other hand, the forgetful functor $V_H : \HRep(T) \to |T|$ is conservative.

Let $f_0$ be a morphism of $\C$ such that $T_a(f)$ is an isomorphism. Define a morphism $\varphi$ of $|T|$ by setting  $\varphi_{i_0} = f$ and  $\varphi_i = \id_\un$ if $i \neq i_0$.
Then $V_H(\varphi^H)$ is an isomorphism, and so is $\varphi$ itself, which implies that $f = \varphi_{i_0}$ is an isomorphism. Thus $T_a$ is conservative.

Now let us show (ii) $\implies$ (i). We verify that the  two functors: $\mathfrak{h}^l : |T| \to \HRep(T)$ and $?^\coinv : \HRep(T) \to |T|$ are inverse to one another.

Let $X$ be an object of $|T|$.
Then $(\mathfrak{h}^l X)^{\coinv}=((\mathfrak{h}^l{X_i})^\coinv)_{i \in D_0}$, and, in view of the initial remark, the right $T_i$-Hopf module $\mathfrak{h}^l X_i$ has coinvariant part $X_i$, hence
$$(\mathfrak{h}^l X)^{\coinv} \simeq X.$$

Let $W = (W,\rho,\delta)$ be a right Hopf representation of $T$. Denote by $\iota_j$ the equalizer morphism ${W^\coinv}_j \to T_jW_j$.
Let $j \labelto{a} i$ be a morphism of $D$, and consider the morphism
$$\theta_a = \rho_{a,j}{\mu_{a,j}}_{W_j} T_a(\iota_j) : T_a W^\coinv_j \to W_a.$$

One verifies that $\theta = (\theta_a)_{a \in D_1}$ is a morphism $\mathfrak{h}^l(W^\coinv) \to W$ in $\HRep(T)$.
On the other hand, since $T$ is left Hopf, the morphism
$$( T_a\un\otimes\rho_{a,j})T^2_a(\un, W_j) : T_a(W_j) \to T_a\un \otimes W_a$$
is an isomorphism. Composing its inverse with $\delta_a : W_a \to T_a\un\otimes W_a$, we obtain a morphism $\xi_a : W_a \to T_a W_j$.
One verifies that $\xi_a$ equalizes the image by $T_a$ of the pair $( {\eta_j}_\un \otimes W_j , \delta_j)$, whose equalizer is $W^\coinv_j$, and since $T_a$ preserves such equalizers,
$\xi_a$ factors uniquely through $T_a(\iota_j)$, defining a morphism $\tilde{\xi}_a : W_a \to T_a(W^\coinv_j)$. One verifies that $\tilde{\xi}_a$ is inverse to $\theta_a$, which shows
$$\mathfrak{h}^l(W^{\coinv}) \simeq W,$$ hence the theorem.
\end{proof}

\section{Hopf polyalgebras, representable Hopf polyads, and Hopf categories in the sense of Batista, Caenepeel, Vercruysse}\label{sect-hopf-cats}

In this section, we interpret Hopf Categories, as defined in \cite{BCV}, as special cases of Hopf polyads.

\subsection{Hopf polyalgebras}
Let $\V$ be a monoidal category.
A \emph{$D$-polyalgebra in $\V$} is a set of data $M = (M,m,u)$, where $M = (M_a)_{a \in D_1}$ is a family of objects of $\V$ indexed by $D_1$,
$m = (m_{a,b})_{(a,b) \in D_2}$ and $u = (u_i)_{i \in D_0}$ are families of morphisms in $\V$, $m_{a,b} : M_a \otimes M_b \to M_{ab}$, and $u_i : \un \to M_i$, satisfying:
$$m_{ab,c}(m_{a,b} \otimes  M_c) = m_{a,bc}(M_a \otimes m_{b,c}) \blabla{and} m_{a,i}(M_a \otimes u_i).$$

\begin{exa} Assume $D = *_X$, where $X$ is a set (recall that the objects of $*_X$ are the elements of $X$, with exactly one morphism between any two objects). Then
$D$-polyalgebras over $\V$ are nothing but $\V$-enriched categories with set of objects $X$.
\end{exa}

Now assume that $\V$ is braided, with braiding $\tau$. In that case, the category $\Coalg(\V)$ of coalgebras in $\V$ is monoidal.

A \emph{$D$-polybialgebra in $\V$} is a $D$-polyalgebra in $\Coalg(V)$.
We may view it as a set of data $M = (M,m,u,\Delta,\eps)$, where $(M,m,u)$ is a $D$-polyalgebra in $\V$ and $\Delta = (\Delta_a)_{a \in \D_1}$, $\eps = (\eps_a)_{a \in D_1}$
are collections of morphisms such that $(M_a,\Delta_a,\eps_a)$ is a coalgebra for each $a \in D_1$, and the morphisms $\mu_{a,b}: M_a \otimes M_b \to M_{a,b}$ and $\eta_i : \un \to M_i$
are coalgebra morphisms.

The \emph{left} and \emph{right fusion operators} of a $D$-polyalgebra $H^l$ and $H^r$ are families of morphisms indexed by $D_2$ defined as follows:
$$H^l_{a,b} = (M_a \otimes m_{a,b})(\Delta_a \otimes M_b) \blabla{and} H^r_{a,b} = (m_{a,b} \otimes M_a)(M_a \otimes \tau_{M_a,M_b})(\Delta_a \otimes M_b).$$

A \emph{$D$-Hopf polyalgebra} is a $D$-polybialgebra whose fusion operators are isomorphisms.

\begin{exa}
For $D = *_X$,
$D$-polybialgebras are exactly $\Coalg(\V)$-enriched categories,
and $D$-Hopf polyalgebra are exactly Hopf categories in the sense of Batista, Caenepeel, Vercruysse \cite{BCV}.
\end{exa}

%

\pagebreak[5]
\subsection{Representable Hopf polyads}

Let $\V$ be a monoidal category.
The assignment $X \mapsto X \otimes ?$ defines an embedding of $\V$ as a monoidal subcategory of $\EndFun(\V)$. Let us call endofunctors of the form $X \otimes ?$ \emph{representable endofunctors}, and natural transformations of the form $f \otimes ?$ between such endofunctors, \emph{representable transformations}. (More precisely, one should rather say `representable on the left', but we will always consider the left-handed case).

A \emph{representable polyad over $\V$} is a polyad of the form $(D,\V^{D_0}, T, \mu, \eta)$ where the functors  $T_a : \V \to \V$ ($a \in D_1$) and the natural transformations
$\mu_{a,b}$ ($(a,b) \in D_1$) and $\eta_i$ ($i \in D_0)$ are representable.

Representable polyads can be represented by polyalgebras.

\begin{lem}
Given a $D$-polyalgebra in $\V$, we define a representable polyad $M \otimes ?$ over $\V$ by
$$M \otimes ? = (D,\V^{D_0},(M_a \otimes ?)_{a \in D_1}, (m_{a,b} \otimes ?)_{(a,b) \in D_2},(u_i \otimes ?)_{i \in D_0}).$$
The assignment $M \mapsto M \otimes ?$ maps polyalgebras bijectively to representable polyads.
\end{lem}

Now assume that $\V$ is braided, with braiding $\tau$.

Given a coagebra $(C,\Delta,\eps)$ in $\V$, the representable endofunctor $C \otimes ?$ admits a comonoidal structure,
given by
\begin{align*}
&(C \otimes ?)^2_{X,Y} = (C \otimes \tau_{C,X} \otimes Y)(\Delta \otimes X \otimes Y): C \otimes X \otimes Y \to C \otimes X \otimes C \otimes Y,\\
&(C \otimes ?)^0 = \eps.
\end{align*}

Such comonoidal endofunctors are called \emph{representable comonoidal endofunctors}.

A comonoidal polyad, or a Hopf polyad, of the form $(D,\V, T, \mu, \eta)$ is \emph{representable} if it is representable as a polyad and the $T_a$'s are representable comonoidal endofunctors.

Representable comonoidal polyads and Hopf polyads can be represented by polybialgebras and Hopf polyalgebras.

\begin{lem}\label{lem-polypoly}
Given a $D$-polyalgebra $M$ in $\V$, the data 
$$M \otimes ? = (D,\V^{D_0},(M_a \otimes ?)_{a \in D_1}, (m_{a,b} \otimes ?)_{(a,b) \in D_2},(u_i \otimes ?)_{i \in D_0})$$
is a representable comonoidal polyad over $\V$,
the comonoidal structure on $M_a \otimes ? : \V \to \V$ being defined by the coalgebra structure of $M_a$.

Moreover, the assignment $M \mapsto M \otimes ?$ maps polybialgebras bijectively to representable comonoidal polyads, and
in this correspondance, Hopf polyalgebras correspond 1-1 with representable Hopf polyads.
\end{lem}

\begin{exa}
 Let $X$ be a set. A $X_*$-polybialgebra, or $\Coalg(\V)$-enriched category with set of objects $X$, corresponds via Lemma~\ref{lem-polypoly} to a representable comonoidal polyad over $\V$, and in this dictionary,
 a Hopf category in the sense of \cite{BCV} corresponds with a representable Hopf polyad over $\V$. Note that the notion of a Hopf module over a Hopf category considered in  \cite{BCV}
corresponds with our notion of a Hopf representation of a Hopf polyad.
\end{exa}

\pagebreak[5]

\section{The center construction}\label{sect-center}

We construct a Hopf polyad associated with a graded tensor category, and, by applying to it the fundamental theorem of Hopf polyads, we recover
a result about the center of a graded tensor category due to Turaev and Virelizier \cite{TV}.

\subsection{The centralizer of a graded category}

Let $\kk$ be a field. A \emph{bounded $\kk$-linear abelian category} is a $\kk$-linear abelian category $\A$ such that there exists a finite dimensional $\kk$-algebra $A$ and a $\kk$-linear equivalence of categories $\A \simeq \mod(A)$ (where $\mod(A)$ denotes the category of finite-dimensional right $\kk$-modules).

A \emph{$\kk$-tensor category}, or just \emph{tensor category}, is a rigid monoidal category $\C$ endowed with a structure of $\kk$-linear abelian category such that the tensor product is linear in each variable,
$\End(\un) = \kk$, $\Hom$-spaces are finite dimensional and all objects have finite length.

%

%
%
%
%
%

Let $G$ be a group. A \emph{$G$-graded tensor category} is a tensor category $\C$ endowed with a family $(\pi_g)_{g \in G}$ of natural endomorphisms
of $\id_\C$, satisfying the following conditions:
\begin{enumerate}[1)]
\item for $X \in \Ob(\C)$, the set $\{g \in G \mid \pi_g(X) \neq 0\}$ is finite, and: $$\sum_{g \in G} \pi_g(X) = \id_X,$$
\item the $\pi_g$'s are orthogonal idempotents, \emph{i.e.} for $g,h \in G$, $\pi_g \pi_h = \delta_{g,h} \pi_g$,
\item for $g, h \in G$ and $X,Y \in \Ob(\C)$, $\pi_{gh}(X\otimes Y) (\pi_g(X) \otimes \pi_h(Y)) = \pi_g(X) \otimes \pi_g(Y)$,
\item $\pi_1(\un) = \id_\un$;
\end{enumerate}
If $\C$ is a $G$-graded tensor category, we say that an object $X$ of $\C$ is \emph{homogeneous of degree $g$} if $\pi_g(X) = \id_X$. We denote by $\C_g$ the full abelian subcategory of $\C$ of homogeneous objects of degree $g$. Then (1) states that any object is a finite direct sum of homogeneous objects of different degrees, (2), that such a decomposition is essentially unique,
(3), that $\C_g \otimes \C_h \subset \C_{gh}$, and (4), that $\un$ is homogeneous of degree $1$.

We say that the $G$-graded tensor category $\C$ is locally bounded if for each $g$, $\C_g$ is bounded.

Let $G$ be a group and let $\C$ be a locally bounded $G$-graded tensor category.
For $X \in \Ob(\C)$ and $g \in G$ the coend $Z_g(X) = \int^{Y \in \C_g} Y \otimes X \otimes \rdual{Y}$  exists, because $\C_g$ is bounded.
Denote by $j^g_{X,Y} : Y \otimes X \otimes \rdual{Y} \to Z_g(X)$ ($Y \in \Ob(\C_g)$) its universal dinatural transformation.

The assignment $X \mapsto Z_g(X)$ defines an endofunctor of $\C$, denoted by $Z_g$.
Define further: 
$${\mu_{g,h}}_X : Z_g Z_h(X) \to Z_{gh}X, \qquad \eta_X : X \to Z_1(X),$$
$$Z^2_g(X_1,X_2) : Z_g(X_1 \otimes X_2)  \to Z_g(X_1) \otimes Z_g(X_2), \qquad Z^0_g : Z_g\un \to \un,$$
where $X, X_1, X_2$ are objects of $\C$ and $g,h$ elements of $G$, by the formulae:
$${\mu_{g}}_X j^g_{X,Y_1}(Y_1 \otimes j^g_{X,Y_2} \otimes \rdual{Y_1}),\qquad \eta_X = j^1_{X, \un}, \qquad Z^0_g j_{\un,Y} = e_Y,$$
$$Z^2_g(X_1,X_2) j^g_{X_1 \otimes X_2,Y} = (j^g_{X_1,Y} \otimes j^g_{X_2,Y})(Y \otimes X_1 \otimes h_Y \otimes X_2 \otimes \rdual{Y}),$$
where $e_Y : Y \otimes \rdual{Y} \to \un$ is the evaluation, and $h_Y : \un \to \rdual{Y} \otimes Y$ the coevaluation of the duality $(Y,\rdual{Y})$.
Pictorially:
\begin{center}
\includegraphics[width=1\textwidth]{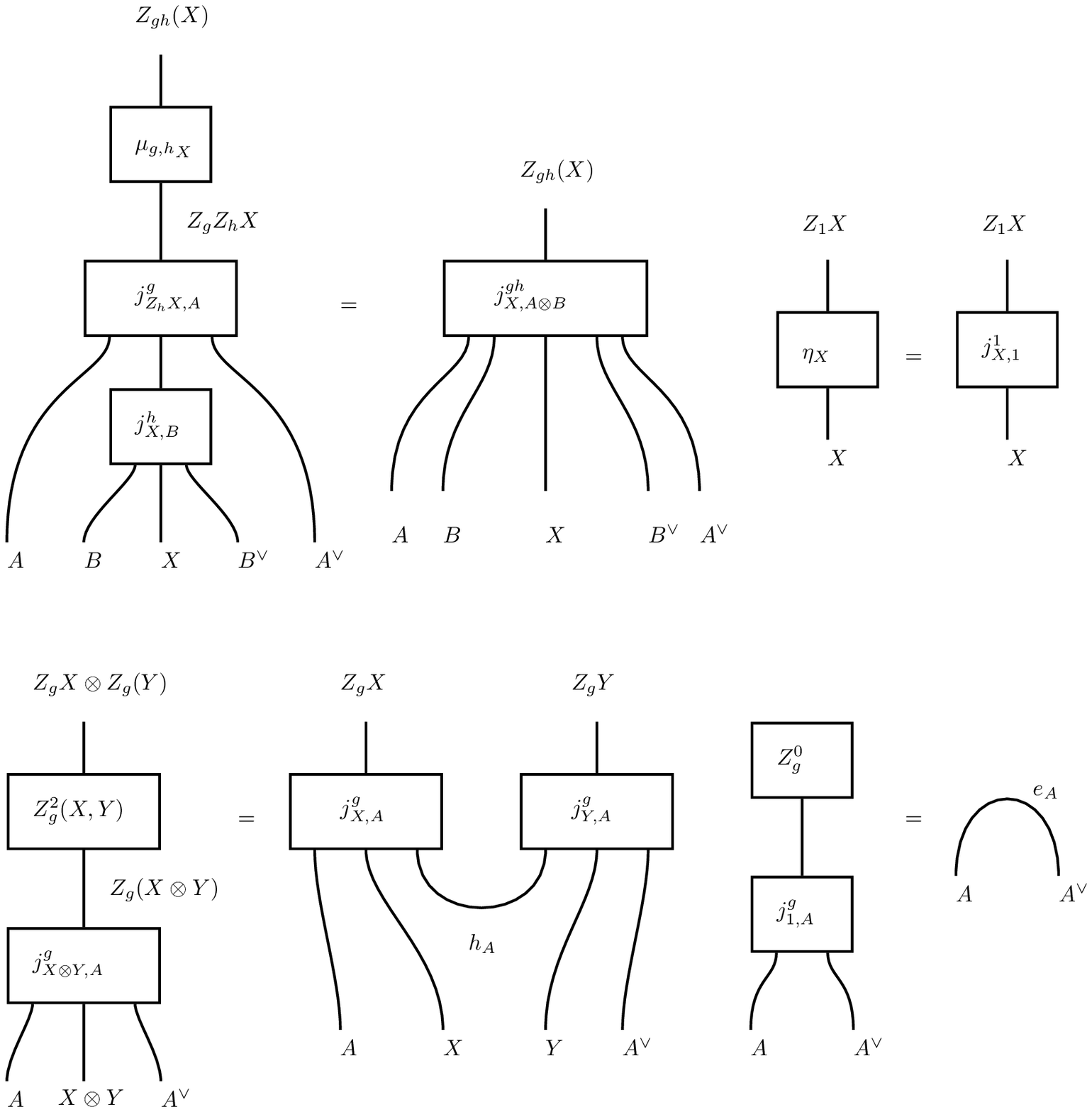}
\label{fig1}
\end{center}

\begin{thm}\label{thm-center-graded}
Let $G$ be a group and let $\C$ be a locally bounded $G$-graded tensor category. Then:
\begin{enumerate}[(1)]
\item The data $Z = (\ul{G},\C,(Z_g)_{g \in G},(\mu_{g,h})_{(g,h) \in G^2},\eta)$ is a \re-exact Hopf polyad;
\item $\Mod(Z)$ is isomorphic to the center $\Z(\C)$, and $\C_{Z_1}$ is isomorphic to the relative center $\Z_{\C_1}(\C)$ as tensor categories.
\end{enumerate}
Assume moreover that for every $g \in G$, $\C_g \neq 0$. Then:
\begin{enumerate}[(1)]
\setcounter{enumi}{2}
\item $\tilde{Z}$ is a strong comonoidal action of $G$ on $\C_{Z_1}$, hence $G$ acts on $\Z_{\C_1}(\C)$;
\item  $\Z(\C)$ is equivalent to the equivariantization of $\Z_{\C_1}(\C)$ under this action of $G$ as a tensor category.
\end{enumerate}

\end{thm}

\begin{defi}The Hopf polyad $Z$ so constructed is called the \emph{centralizer of the $G$-graded tensor category $\C$}.
\end{defi}

\begin{proof}[Proof of Theorem~\ref{thm-center-graded}]
The fact that $Z$ is a comonoidal polyad can be checked easily (The proof can be adapted  from that of \cite{BV3}, Theorem 5.6.)
Moreover $Z$ is \re-exact because $\C$ is a tensor category and each $Z_g$ is right exact (as inductive limit of exact functors).

Now let us prove that $Z$ is a left Hopf polyad (the right-handed version ensues by virtue of symmetry).
We have $$Z_g(X \otimes Z_h Y) = \int^{A \in \C_g, B  \in C_h} A \otimes X \otimes B \otimes Y \otimes \rdual B \otimes \rdual{A},$$
$$Z_g X \otimes Z_k Y = \int^{A \in \C_g, C \in \C_k} A \otimes X \otimes \rdual{A} \otimes C \otimes Y \otimes \rdual{C},$$
with universal dinatural morphisms
$$I^{g,h}_{X,Y,A,B} : A \otimes X \otimes B \otimes Y \otimes \rdual{B} \otimes \rdual{A} \to Z_g(X \otimes Z_h(Y),$$
$$J^{g,k}_{X,Y,A,C} :  A \otimes X \otimes \rdual{A} \otimes C \otimes Y \otimes \rdual{C} \to  Z_g X \otimes Z_k Y$$ defined by:
$$I^{g,h}_{X,Y,A,B} = j^g_{X \otimes Z_a Y,A}(A \otimes j^h_{Y,B} \otimes \rdual{A}) \blabla{and} J^{g,k}_{X,Y,A,C}= j^g_{X,Y,A,C},$$
pictorially:
\begin{center}
\includegraphics[width=1\textwidth]{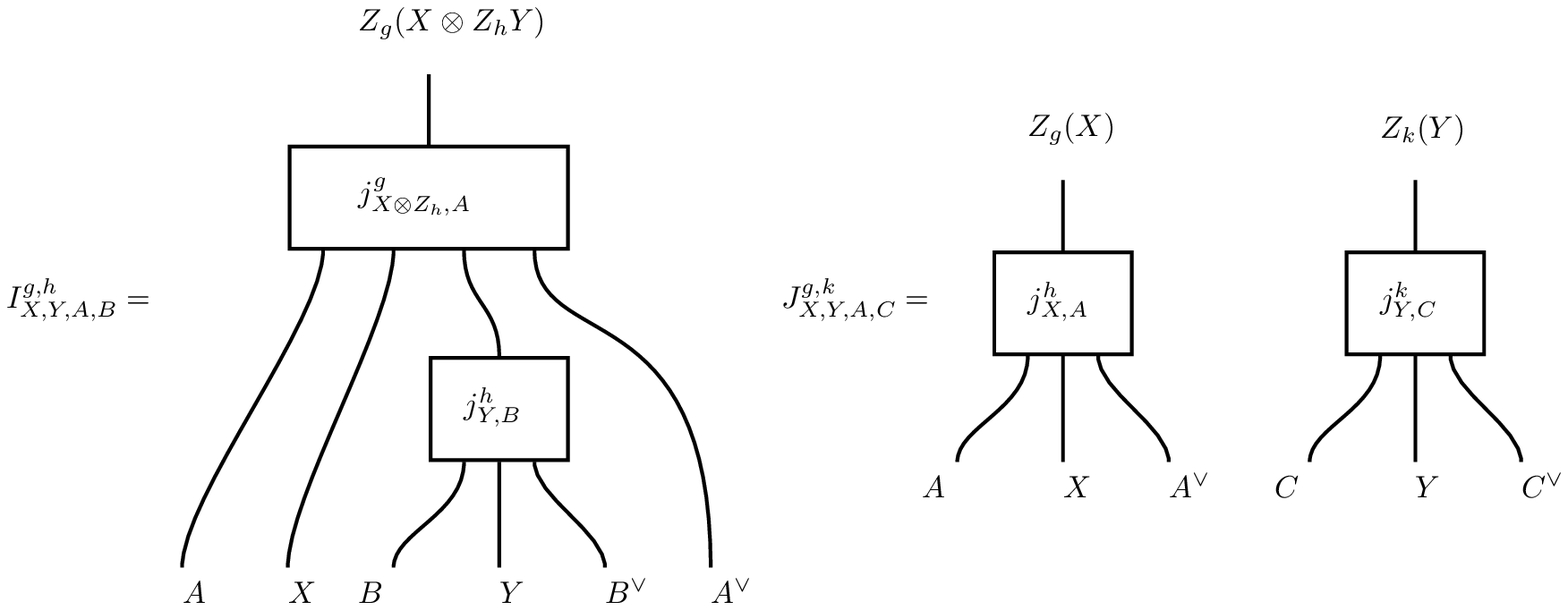}
\label{fig2}
\end{center}
The left fusion morphism $H^l_{g,h}(X,Y)$ can be described by means of $I$:
\begin{center}
\includegraphics[width=.7\textwidth]{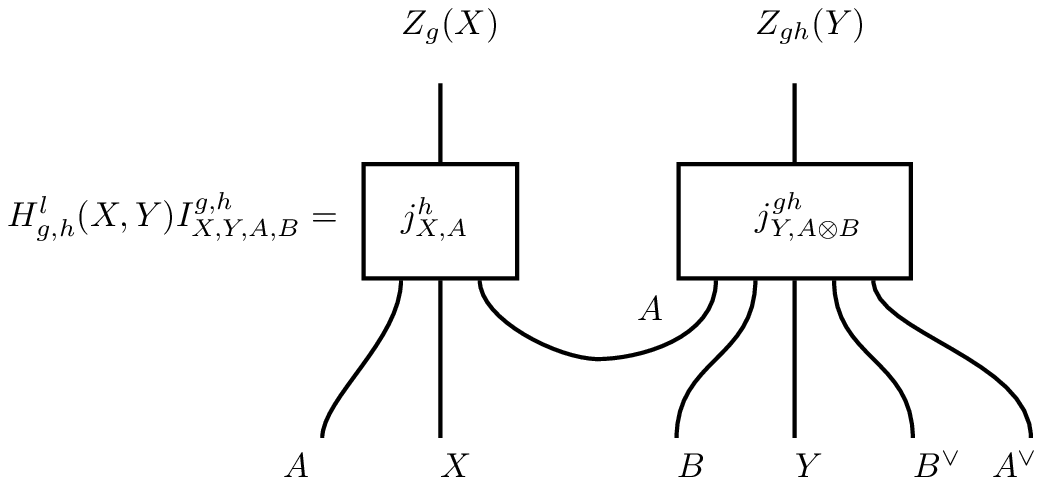}
\label{fig3}
\end{center}
that is, $H^l_{g,h}(X,Y) I^{g,h}_{X,Y,A,B}= J^{g,gh}_{X,Y,A,A \otimes B}(A \otimes X \otimes h_A \otimes B \otimes Y \otimes \rdual{B} \otimes \rdual{A})$.

We define a new morphism $K^l_{g,h}(X,Y): Z_g(X) \otimes Z_{gh}(Y) \to Z_g(X \otimes Z_hY))$ as follows:
\begin{center}
\includegraphics[width=.7\textwidth]{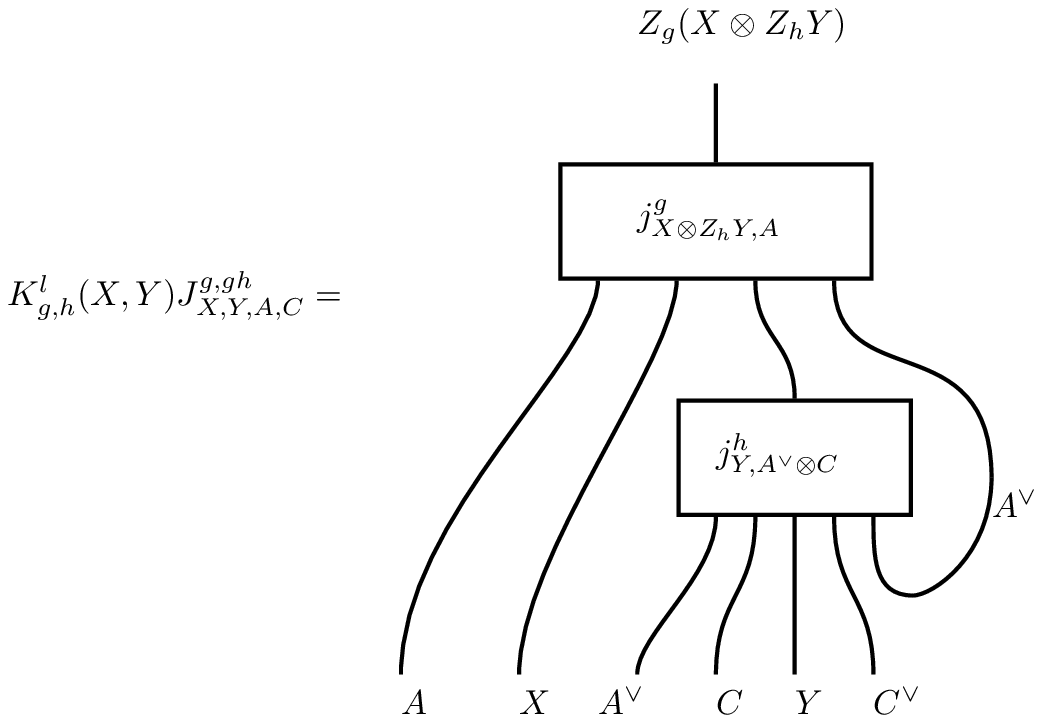}
\label{fig4}
\end{center}
that is, $K^l_{g,h}(X,Y) J^{g,gh}_{X,Y,A,C}=I^{g,h}_{X,Y,A,\rdual{A}  \otimes C}(A \otimes X \otimes \rdual{A} \otimes C \otimes Y \otimes \rdual{C} \otimes h_\rdual{A})$;
and we now verify that it is the inverse of the left fusion operator:
\begin{align}
K^l_{g,h}&(X,Y)H^l_{g,h}(X,Y) I^{g,h}_{X,Y,A,B} \cr
&= K^l_{g,h}(X,Y)J^{g,gh}_{X,Y,A,A \otimes B}(A \otimes X \otimes h_A \otimes B \otimes Y \otimes \rdual{B} \otimes \rdual{A}) \cr
&=I^{g,h}_{X,Y,A,\rdual{A} \otimes A \otimes B}(A \otimes X \otimes h_A \otimes B \otimes Y \otimes \rdual{B} \otimes  \rdual{A} \otimes h_\rdual{A}) = I^{g,h}_{X,Y,A,B},\nonumber
\end{align}
so that $K^l_{g,h}(X,Y)$ is left inverse to $H^l_{g,h}(X,Y)$, and similarly:
\begin{align}
H^l_{g,h}&(X,Y) K^l_{g,h}(X,Y) J^{g,gh}_{X,Y,A,C} \cr
&= H^l_{g,h}(X,Y) I^{g,h}_{X,Y,A,\rdual{A}  \otimes C}(A \otimes X \otimes \rdual{A} \otimes C \otimes Y \otimes \rdual{C} \otimes h_\rdual{A})\cr
&=J^{g,gh}_{X,Y,A,A\otimes \rdual{A} \otimes C}
(A \otimes X \otimes h_A \otimes A \otimes C \otimes Y \otimes \rdual{C} \otimes h_{\rdual{A}}) = J^{g,gh}_{X,Y,A,C},
\nonumber\end{align}
which shows that $K^l_{g,h}(X,Y)$ is right inverse to $H^l_{g,h}(X,Y)$.
This concludes the proof of Assertion (1).

Let us show Assertion (2). Let $X$ be an object of $\C$. A natural transformation $\sigma :  X \otimes \id_\C \to \id_\C \otimes X$ can be viewed (by duality) as a dinatural transformation
$r_a : A \otimes X \otimes \rdual{A} \to X$, with $A$ in $\C$, or, equivalently, as a collection $(r_g)_{g \in G}$ of dinatural transformations $r^g_A : A \otimes X \otimes \rdual{A} \to X$, with $A$ in $\C_g$. In turn, each dinatural transformation $r_g$ translates into a morphism $\rho_g: Z_g X \to X$. Let $\rho = (\rho_g)_{g \in G}$.

One verifies easily that in this correspondance, $(X,\sigma)$ is a half-braiding of $\C$ if and only if $(X,\rho)$ is a $Z_g$-module, and so defines an isomorphism of monoidal
categories: $\Z(\C) \iso \Mod(Z)$.

Similarly a half-braiding $(X,\sigma)$ relative to the inclusion $\C_1 \subset \C$ translates into a $Z_1$-module $(X,\rho_1 : Z_1 X \to X)$, hence an isomorphism of monoidal categories
$\Z_{\C_1}(\C) \iso \C_{Z_1}$.

Assertion (3) boils down to Theorem~\ref{thm-fond} applied to $Z$.  All we have to check is that $Z$ is transitive. Now by assumption for any $g \in G$, $\C_g$ contains a non-zero object $X$, and $Z^0_g j^g_{\un,X} = e_X$ is non-zero, which implies that $Z_g(\un)$ is non-zero, and  $Z_g(\un) \otimes ?$ is conservative because it is faithful exact and $\C$ is abelian.
As a result, $\tilde{Z}$ is strong comonoidal of action type, that is, it is an action of the group $G$ on $\C_{Z_1} \simeq \Z_{\C_1}(\C)$.
By Proposition~\ref{prop-comopo}, $\Mod(\tilde{Z}) \simeq \Mod(Z)$, which means, in view of Assertion (2), that the equivariantization of $\C_{Z_1} \simeq \Z_{\C_1}(\C)$ under the group action of $G$ is nothing but $\Z(\C)$, that is, Assertion (4) holds.
\end{proof}

\begin{rem}\label{rem-ex-centralizer}
Hopf polyads can also be constructed in the following way: let $D$ be a small category, and consider a data $$(\M= ((\M_a)_{a \in D_1},\otimes = (\otimes_{a,b})_{(a,b) \in D_2},\un=(\un_i)_{i \in D_0}),$$ where $\M_a$ is a category, $\otimes_{a,b} : \M_a \times \M_b \to \M_{ab}$ is a functor, $\un_i$ is an object of $\M_i = \M_{\id_i}$, such that $\mu$ and $\un$ satisfy the expected associativity and unity condition. (A lax version involves associativity and unity constraints.)
Then the categories $\M_i$ are monoidal, and for $j\labelto{a} i$ in $D_1$, $\M_a$ is a $\M_j$-$\M_i$ bimodule category.
Under suitable hypotheses (existence of certain adjoints and coends), one defines for each morphism $j\labelto{a} i$ two functors:
$$\left\{
\begin{array}{llll}
 Z^l_a : &\M_i &\to &\M_j\\
  &X &\mapsto &\int^{A \in \M_a} [A,A \otimes X]_l
\end{array}
\right.
\left\{
\begin{array}{llll}
 Z^r_a : &\M_j &\to &\M_i\\
  &Y &\mapsto &\int^{A \in \M_a} [A,Y \otimes A]_r
\end{array}
\right.$$
where $[-,-]_l$ and $[-,-]_r$ denote left and right external cohoms.
Hence two quasitriangular Hopf polyads $Z^l$ and $Z^r$, with source $D^o$ and $D$ respectively. 
\end{rem}

\bibliographystyle{amsalpha}

\begin{thebibliography}{A}




%


\bibitem[BW85]{BaWe}
M.~Barr and Ch.~Wells, \emph{Toposes, triples and theories}, Springer-Verlag (1985), republished in: Repr. Theory Appl. Categ. 12 (2005), 1--287.


\bibitem[BCV15]{BCV} Hopf Categories
E. Batista, S. Caenepeel, J. Vercruysse, \emph{Hopf Categories}, 	arXiv:1503.05447 [math.QA] (2015), 36 p.


\bibitem[Bec69]{Beck1}
J.~Beck, \emph{Distributive laws}, Sem. on Triples and Categorical Homology
  Theory (ETH, Z\"urich, 1966/67), Springer, Berlin (1969), pp.~119--140.



%

\bibitem[BN10]{BN1}
A.~Brugui{\`e}res and S.~Natale,  \emph{Exact Sequences of Tensor Categories}, Int Math Res Notices (2011) 24: 5644-5705.

\bibitem[BN11]{BN2}
\bysame,  \emph{Central exact sequences of tensor categories, equivariantization and applications},     J. Math. Soc. Japan Volume 66, Number 1 (2014), 257-287.



\bibitem[BV07]{BV2}
A.~Brugui{\`e}res and A.~Virelizier, \emph{Hopf monads}, Advances in Math. 215 (2007), 679--733.


\bibitem[BV09]{BV3}\bysame, \emph{Quantum Double of Hopf monads and Categorical Centers}, Trans. Amer. Math. Soc., 364 (2012), 1225-1279.
, avec Alexis Virelizier - (pdf) Trans. Amer. Math. Soc.

\bibitem[BLV11]{BLV} A. Brugui{\`e}res, S. Lack, A. Virelizier; \emph{Hopf monads on monoidal categories}, Advances in Math. 227 (2011), 745-800.



\bibitem[DPS07]{DPS}
B.~Day, E.~Panchadcharam, and R.~Street, \emph{Lax braidings and the lax centre}, Hopf algebras and generalizations, Contemporary
  Mathematics, 2007.


\bibitem[DS07]{DayStreet}
B.~Day and R.~Street, \emph{Centres of monoidal categories of functors},
  Categories in Algebra, Geometry and Mathematical Physics, Contemporary
  Mathematics, 2007.







%







\bibitem[{Mac}98]{ML1}
S.~{Mac Lane}, \emph{Categories for the working mathematician}, second ed.,
  Springer-Verlag, New York (1998).

\bibitem[McC02]{MCC}
P.~McCrudden,  \emph{Opmonoidal monads}, Theory Appl. Categ. 10 (2002), 469--485.

\bibitem[Maj94]{Maj3}
S.~Majid,  \emph{Crossed products by braided groups and bosonization}, J. Algebra 163 (1994), 165--190.

\bibitem[Maj95]{Maj2}
\bysame, \emph{Foundations of quantum group theory}, {Cambridge: Cambridge
  Univ. Press. xix, 607 p.}, 1995.


\bibitem[Moe02]{Moer}
I.~Moerdijk, \emph{Monads on tensor categories}, J. Pure Appl. Algebra
  168 (2002), no.~2-3, 189--208.


\bibitem[TV10]{TV0} V. Turaev and A. Virelizier, \emph{On two approaches to 3-dimensional TQFTs}, preprint arXiv:1006.3501 (2010).

\bibitem[TV13]{TV} V. \bysame, \emph{
 On the graded center of graded categories}, J. Pure Appl. Algebra 217 (2013), 1895-1941

%
























\end{thebibliography}

\end{document}